\documentclass[final,3p,times]{article}

\usepackage[a4paper,left=25mm,top=25mm,right=20mm,bottom=35mm]{geometry}

\usepackage{authblk}
\usepackage[numbers, sort&compress]{natbib}

\usepackage{color}

\usepackage[colorlinks]{hyperref}
\hypersetup{
	citecolor=blue,
   linkcolor=red,
}

\usepackage{amsmath,amsopn,amsthm,amssymb}%,mathabx}

\usepackage[margin=30pt,font=small,labelfont=bf]{caption}

\usepackage{graphicx}
\usepackage{mathptmx} 
\usepackage[mathscr]{euscript}
\usepackage{natbib}
\usepackage{mathtools}
\usepackage{amssymb}
\usepackage{wasysym}
\usepackage{float}
\usepackage{enumitem}
\usepackage{xspace}

\usepackage{amsthm}
\newtheorem{theorem}{Theorem}[section]
\newtheorem{definition}[theorem]{Definition}
\newtheorem{fact}[theorem]{Observation}
\newtheorem{lemma}[theorem]{Lemma}
\newtheorem{proposition}[theorem]{Proposition}
\newtheorem{corollary}[theorem]{Corollary}
\newtheorem{claim}{Claim}[theorem]
\newenvironment{claim-proof}%
{\begin{description}[noitemsep,topsep=0pt,leftmargin = 0.2cm, labelsep = 0.2cm]
\item \emph{Proof of Claim.}}{\hfill$\diamond$\end{description}}
\newtheorem{remark}[theorem]{Remark}

\usepackage[noend]{algpseudocode} % with option noend "end for" etc. is omitted
\usepackage{algorithm}
\DeclareMathOperator{\lca}{lca}
\DeclareMathOperator{\med}{med}
\DeclareMathOperator{\child}{child}

\newcommand{\M}{\ensuremath{\mathbb{M}}}
\newcommand{\Mmax}{\ensuremath{\mathbb{M}_{\max}}}
\newcommand{\Ms}{\ensuremath{\mathbb{M}_{\mathrm{str}}}}
\newcommand{\Hext}{\ensuremath{H^{\mathrm{ext}}}}
\newcommand{\Qext}{\ensuremath{Q^{\mathrm{ext}}}}
\newcommand{\Xirr}{\ensuremath{X^{2}_\mathrm{irr}}}
\newcommand{\Lirr}{\ensuremath{L^{2}_\mathrm{irr}}}

\providecommand{\keywords}[1]{\textbf{\textit{Keywords: }} #1}

\title{From Modular Decomposition Trees to 
			 Rooted Median Graphs}

 \author[1]{Carmen Bruckmann} 
\author[2-7]{Peter F. Stadler}    
\author[8,*]{Marc Hellmuth} 
\affil[1]{Helmholtz-Centre for Environmental Research, Permoserstraße
  15, D-04318 Leipzig, Germany}

\affil[2]{Bioinformatics Group, Department of Computer Science \&
    Interdisciplinary Center for Bioinformatics, Universit{\"a}t Leipzig,
    H{\"a}rtelstra{\ss}e~16--18, D-04107 Leipzig, Germany.}

\affil[3]{German Centre for Integrative Biodiversity Research
  (iDiv) Halle-Jena-Leipzig, Competence Center for Scalable Data Services
  and Solutions Dresden-Leipzig, Leipzig Research Center for Civilization
  Diseases, and Centre for Biotechnology and Biomedicine at Leipzig
  University at Universit{\"a}t Leipzig}

\affil[4]{Max Planck Institute for Mathematics in the Sciences,
  Inselstra{\ss}e 22, D-04103 Leipzig, Germany} 

\affil[5]{Institute for Theoretical Chemistry, University of Vienna,
  W{\"a}hringerstrasse 17, A-1090 Wien, Austria}

\affil[6]{Facultad de Ciencias, Universidad National de Colombia, Sede
  Bogot{\'a}, Colombia}

\affil[7]{Santa Fe Insitute, 1399 Hyde Park Rd., Santa Fe NM 87501,
  USA}

\affil[8]{Department of Mathematics, Faculty of Science,
  Stockholm University, SE-10691 Stockholm, Sweden 
  \newline \texttt{mhellmuth@mailbox.org}}

\affil[*]{corresponding author}

\date{\ }

\begin{document}
\sloppy

\maketitle

\abstract{ 
  The modular decomposition of a symmetric map
    $\delta\colon X\times X \to \Upsilon$ (or, equivalently, a set of
    symmetric binary relations, a 2-structure, or an edge-colored
    undirected graph) is a natural construction to capture key features of
    $\delta$ in labeled trees. A map $\delta$ is explained by a
    vertex-labeled rooted tree $(T,t)$ if the label $\delta(x,y)$ 
    coincides with the label of the last common ancestor of $x$ and $y$ in
    $T$, i.e., if $\delta(x,y)=t(\lca(x,y))$. Only maps whose modular
    decomposition does not contain prime nodes, i.e., the symbolic
    ultrametrics, can be exaplained in this manner. Here we consider rooted
    median graphs as a generalization to  (modular decomposition) trees to
    explain symmetric maps.  We first show that every symmetric map can be
    explained by ``extended'' hypercubes and half-grids. We then derive a a
    linear-time algorithm that stepwisely resolves prime vertices in the
    modular decomposition tree to obtain a rooted and labeled median graph
    that explains a given symmetric map $\delta$. We argue that the
    resulting ``tree-like'' median graphs may be of use in phylogenetics as
    a model of evolutionary relationships.
}


%  
%\end{frontmatter}
\section{Introduction}

The decomposition of an object comprising a finite point set $X$ into its
modules has been a topic of intense research for decades, starting with
Tibor Gallai's seminal work \cite{Gallai:67}. A module is a subset
$M\subseteq X$ such that the points within $M$ cannot be distinguished from
each other in terms of their relationships with points in $X\setminus M$.
Modular decompositions have been studied for graphs
\cite{CP-06,Corneil:81,EHMS:94,McConnell:95,habib2010survey,HFWS:20},
labeling maps or equivalently sets of binary relations and 2-structures
\cite{Boecker:98,HW:16a,HW:15,HSW:16,ER1:90,ER2:90,EHR:94,EHPR:96,EHRENFEUCHT1990343,Ehrenfeucht1995,ehrenfeucht1999theory,EHMS:94}
or sets of $n$-ary relations and, equivalently, hypergraphs
\cite{bonizzoni1999algorithm,bonizzoni1995modular,boussairi20183,habib2019general}. They
share fundamental properties irrespective of the type of the objects that
determines the nature of the pertinent relationships. In particular, the
strong modules, i.e., the modules that do not with overlap with other
modules, form a hierarchy and thus can be identified with the vertices of a
rooted tree $T$. This \emph{modular decomposition tree}, whose leaves
correspond to the point set $X$, captures a wealth of information on the
object under consideration.

Models of evolutionary relationships often start from a rooted tree $T$
endowed with labels of vertices or edges that designate evolutionary
events.  A broad class of inverse problems thus arises in mathematical
phylogenetics that can be phrased as follows: Given a (rooted) tree $T$
with labels on vertices and edges and a rule to derive a map $\delta$ on
pairs (or $k$-tuples) of leaves, one asks (i) which maps $\delta$ can be
\emph{explained} by such a labeled tree, (ii) how an explaining labeled
tree $T$ can be constructed, and (iii) to characterize the set of labeled
trees $T$ that explain a given map $\delta$. Not surprisingly, there is a
close connection to the modular decomposition.

An ``event labeling'' $t$ at the inner vertices of $T$, naturally defines
$\delta(x,y):=t(\lca_T(x,y))$ as the label of the last common ancestor of
two leaves $x$ and $y$. An event-labeled tree $(T,t)$ explaining $\delta$
in this manner exists if and only if $\delta$ is a symbolic ultrametric
\cite{Boecker:98}, which shares the structure of a cograph
\cite{Corneil:81}. The corresponding co-tree or symbolic discriminating
representation, i.e., the modular decomposition tree of $\delta$, is
obtained by contracting edges in $(T,t)$ whose endpoints share the same
label \cite{Hellmuth:13a}. In phylogenetic applications, symbolic
ultrametrics describe the key concepts of orthology and paralogy, i.e., the
question whether a pair of related genes descends from a speciation or a
gene duplication event \cite{Hellmuth:13a,Hellmuth:15a,HW:16book}. Edge
weights may model distances or specific types of evolutionary events. In a
\emph{pairwise compatibility graph} (PCG) an edge is drawn whenever the sum
of edge weights between two leaves lies within a specified interval
\cite{Kearney:03,PCGsurvey}. In phylogenetics, they model e.g.\ rare events
\cite{Hellmuth:18b}. Horizontal gene transfer is captured by Fitch graphs
\cite{Geiss:17a,HSS:20,Hel-18,HS:18}, a subclass of directed cographs
corresponding to pairs $(x,y)$ such that an edge with non-zero weight
appears along the path in $T$ connecting $\lca(x,y)$ with the leaf
$y$. Interpreting edge-weights as distances, and thus $\delta(x,y)$ as
distance between leaves, leads to the key theorem of mathematical
phylogenetics: There is a unique edge-weighted tree $T$ if and only if
$\delta$ satisfies the so-called 4-point condition
\cite{SimoesPereira:69,Buneman:71}.

Phylogenetics also motivates the investigation of generalizations. While
trees are an excellent model of many evolutionary systems, they are
approximations and sometimes networks are a better model of reality
\cite{Huson:11}. In the case of distance-based phylogenetics, this
naturally connects with theory of split-decomposable metrics
\cite{Bandelt:92} and their natural representations, the Buneman graphs
\cite{Dress:97,Huber:02}. The latter form a subclass of Median
graphs. Median networks, furthermore, play an important role as
representations of phylogenies within populations
\cite{Bandelt:95,Bandelt:00}. This suggests to consider median graphs, or
subclasses of median graphs such as the Buneman graphs, as a natural
generalization of trees in the context of phylogenetic questions.

In this contribution, we consider rooted median graphs, i.e., median graphs
with a vertex designated as the root. We note in passing that rooted median
graphs recently have attracted in the context of the Daisy graph
construction \cite{Dravec:20,Taranenko:20}. It is natural then to replace
the last common ancestor $\lca(x,y)$ by the median $\med(x,y,\rho)$ of two
vertices $x$ and $y$ and the root $\rho$ and to ask which maps can be
explained as $\delta(x,y)=t(\med(x,y,\rho))$, where $t$ is now a vertex
labeling on the median graph. In Sect.~\ref{sect:general} we show that
every symmetric map can be explained by a sufficiently large median graph
and provide explicit constructions for extended hypercubes and extended
half-grids. Thm.~\ref{thm:median-explains} shows that $O(|X|^2)$ vertices
are sufficient.  In the second part of this contribution,
Sect.~\ref{sect:MDT}, we show that for every symmetric map $\delta$ a
rooted labeled median graph can be obtained by expanding the map's modular
decomposition tree by replacing its vertices with explicitly constructed
graphs (Thm.~\ref{thm:pvr}).  This yields a practical algorithm to
construct a rooted median graph that explains $\delta$ with a running time
linear in the size of the input.  It reduces to a labeling of the modular
decomposition tree exactly for symbolic ultrametrics
(Thm.~\ref{thm:symbolic}).

\section{Preliminaries}

\paragraph{Sets and Maps}
Let $X$ be a finite set. We write
$\Xirr \coloneqq X\times X\setminus \{(x,x)\mid x\in X\}$ for the Cartesian
set product without reflexive elements, $\binom{X}{k}$ for the set of all
$k$-element subsets of $X$, and $2^X$ for the powerset of $X$.  Denoting by
$|X|$ the cardinality of $X$ we have $|\Xirr|=|X|(|X|-1)$.

A \emph{hierarchy on $X$} is a subset $\mathcal{H}\subseteq 2^X$ such that
(i) $X\in\mathcal{H}$, (ii) $\{x\}\in\mathcal{H}$ for all $x\in X$, and
(iii) $p\cap q\in \{p,q,\emptyset\}$ for all $p,q\in\mathcal{H}$.
Condition (iii) states that no two members of $\mathcal{H}$ overlap.

Let $X$ and $\Upsilon$ be
non-empty sets. We consider \emph{maps} $\delta\colon \Xirr \to \Upsilon$
that assign to each pair $(x,y)\in \Xirr$ the unique label
$\delta(x,y) \in \Upsilon$.  A map $\delta$ is \emph{symmetric} if
$\delta(x,y) = \delta(y,x)$ for all distinct $x,y\in X$.  For a subset
$L\subseteq X$ we denote with $\delta_{|L} \colon \Lirr \to \Upsilon$ the
map obtained from $\delta$ by putting $\delta_{|L}(x,y) = \delta(x,y)$ for
all distinct $x,y\in L$.

\paragraph{Graphs}
All graphs $G=(V,E)$ considered here are undirected and simple.  For a
subset $W\subseteq V$, we write $G-W$ for the graph obtained from $G$ by
deleting all vertices in $W$ and their incident edges.

All paths in $G$ are considered to be simple, that is, no vertex is
traversed twice. In particular, the graph $P_n$ denotes the path on $n$
vertices with vertex set $V(P_n) = \{1,\dots,n\}$ and edge set
$E(P_n) = \{ \{i,i+1\} \mid 1\leq i<n\}$. We also write $P_G(a,b)$ for a
path connecting two vertices $a$ and $b$ in $G$.  A cycle is a graph for
which the removal of any edge results in a path.  A cycle of length four is
called a \emph{square}.

The distance $d_G(u,v)$ between vertices $u$ and $v$ in a graph $G$ is the
length $|E(P)|$ of a shortest path $P$ connecting $u$ and $v$. The
\emph{interval} between $x$ and $y$ is the set $I_G(x,y)$ of vertices that
lie on shortest paths $P_G(x,y)$ between $x$ and $y$.

Let $G=(V,E)$ be a graph and $W\subseteq V$. A (partial vertex) labeling is
a map $t\colon W \to \Upsilon$ that assigns to every vertex $v\in W$ one
label $t(v)\in\Upsilon$. We write $(G,t)$ for a given graph $G$ together
with labeling $t$ and call $(G,t)$ \emph{labeled} graph.

\paragraph{Rooted Graphs and Trees}
We consider here \emph{rooted} graphs, that is, graphs for which there is a
particular distinguished vertex $\rho_G\in V$, called the \emph{root of
  $G$}.  Given a rooted graph $G$, we can equip the vertex set $V$ with a
partial order $\preceq_G$ by putting $u\preceq_G v$ whenever $v$ lies on
some path $P_G(\rho_G,u)$ connecting $\rho_G$ and $u$. If $u\preceq_G v$
and $u\neq v$, then we write $u\prec_G v$.  Furthermore, if we have an edge
$\{u,v\}\in E(G)$ such that $u\prec_T v$, then $u$ is a \emph{child} of $v$
and $\child_G(v)$ denotes the set of all children of $v$ in $G$. If
$u\preceq_G v$ or $v\preceq_G u$, the vertices $u$ and $v$ are
\emph{comparable} (in $G$) and \emph{incomparable}, otherwise.

A vertex $v\ne\rho_G$ in a graph $G$ is called \emph{leaf} if its degree
$\deg_G(v) = 1$. The \emph{inner} vertices in a graph $G$ are vertices that
are not leaves and $V^0(G)$ denotes the set of all inner vertices of $G$.
Let $G=(V,E)$ be a graph and assume that we have added the vertex
$x\notin V$ to $G$ such that $x$ is adjacent to exactly one vertex
$v\in V$.  Then, we say that $x$ is \emph{leaf-appended to $v$ (in $G$)}.

A \emph{tree} is a connected acyclic graph. Given a rooted tree $T$, the
last common ancestor $\lca_T(x,y)$ of two vertices $x,y\in V(T)$ is the
unique $\preceq_T$-minimal vertex that satisfies
$x,y\preceq_T \lca_T(x,y)$, that is, there is no further vertex $v$ with
$x,y\preceq_T v\prec_T \lca_T(x,y)$.  For rooted trees $T$ with leaf set
$L$ we denote with $L(v)$ the subset of leaves $x\in L$ with
$x\preceq_T v$. We also write $L_T(v)$ instead of $L(v)$, if there is a
risk of confusion.  Two labeled trees $(T,t)$ and $(T',t')$ with labeling
$t\colon V^0(T)\to \Upsilon$ and $t'\colon V^0(T')\to \Upsilon$ are
\emph{isomorphic} if $T$ and $T'$ are isomorphic via a map
$\psi:V(T)\to V(T')$ such that $t'(\psi(v))= t(v)$ holds for all
$v\in V^0(T)$. There is a well-known bijection between hierarchies and
rooted trees \cite{Semple:03}:
\begin{proposition}
  Let $\mathcal{H}$ be a set of non-empty subsets of $L$. Then,
  there is a rooted tree $T=(W,E)$ with leaf set $L$ and with
  $\mathcal{H} = \{L(v)\mid v\in W\}$ if and only if $\mathcal{H}$ is a
  hierarchy on $L$. Moreover, if there is such a rooted tree, then, up to
  isomorphism, $T$ is unique. 
\label{A:prop:hierarchy}
\end{proposition}

\begin{remark}
  Instead of graphs with leaves we could consider a straightforward
  generalization of the notion of $\mathscr{X}$-trees frequently employed
  in mathematical phylogenetics. There a set of taxa $\mathscr{X}$ is mapped
  (not necessarily injectively) to the vertex set $V(T)$ of a rooted or
  unrooted tree $T$, see e.g.\ \cite{Semple:03}. We prefer to instead
  identify the taxa with the leaf set $L$ and insist that distinct taxa are
  represented by distinct vertices in the graphs that describe the
  phylogenetic relationships. In an $\mathscr{X}$-tree like setting we
  could identify the taxa with the corresponding leaf's (uniquely defined)
  ``parent''.
\end{remark}

\paragraph{Cartesian Graph Product}

The \emph{Cartesian product} $G\Box H$ of two graphs $G=(V,E)$ and
$H=(W,F)$ has vertex set $V(G\Box H)\coloneqq V\times W$ and edge set
$E(G\Box H)\coloneqq \{\{(g,h),(g',h')\} \mid g=g' \text{ and } \{h,h'\}\in
F\text{, or }h=h' \text{ and } \{g,g'\}\in E\}$. The Cartesian product is
known to be commutative and associative and thus,
$\Box_{i=1}^n G_i = G_1\Box \cdots \Box G_n$ is well-defined
\cite{Hammack:2011a,IKR:08}. For a vertex
$v = (g_1,\dots,g_h)\in V(\Box_{i=1}^n G_i)$ we refer to $(g_1,\dots,g_n)$
as the \emph{coordinate vector} of $v$ and to $g_i$ as the \emph{$i$-th
  coordinate} of $v$. The \emph{Hamming distance} between two vertices $v$
and $w$ with coordinate vectors $(g_1,\dots,g_n)$ and $(g'_1,\dots,g'_n)$,
respectively, is the number of coordinates $i$ for which $g_i \neq g'_i$.

A \emph{complete grid} is the Cartesian product $P_n\Box P_m$ of two paths
and a \emph{grid graph} is a subgraph of a complete grid.  A
\emph{hypercube $Q_n$} is the $n$-fold Cartesian product of edges, i.e.,
$Q_n = \Box_{i=1}^n K_2$.  Equivalently, hypercubes can be defined as
graphs having vertex set $V = \{0,1\}^n$ and having edges precisely between
the vertices that have Hamming distance $1$.

Based on the \emph{distance formula} \cite[Cor.\ 5.2]{Hammack:2011a}, we
obtain the following simple result for complete complete grid graphs and
hypercubes that we shall need for later reference.
\begin{lemma}
  Let $P_n$ be a path with $V(P_n) = \{1,\dots,n\}$ and edge set
  $E(P_n) = \{\{i,j\} \mid j=i+1, 1\leq i<n\}$ and $G=P_n\Box P_n$.Then,
  for all vertices $(i,j), (i',j')\in V(G)$ it holds that
  $d_{G}((i,j), (i',j')) = |i-i'|+|j-j'|$.
				
  For a hypercube $Q_n$ the distance $d_{Q_n}(x,y)$ of vertices
  $x,y\in V(Q_n)$ is the Hamming distance of $x$ and $y$.
\label{lem:propertyCart}
\end{lemma}
Following \cite{KS:00}, we will consider grid graphs as plane graphs, that
is, we will assume that they are embedded in the plane in the natural way
-- as a subgraph of a complete grid. 

\paragraph{Median Graphs}

A vertex $x$ is a \emph{median} of a triple of vertices $u$, $v$ and $w$ if
$d(u,x)+d(x,v)=d(u,v)$, $d(v,x)+d(x,w)=d(v,w)$ and $d(u,x)+d(x,w)=d(u,w)$.
A connected graph $G$ is a \emph{median graph} if every triple of its
vertices has a unique median.  In other words, $G$ is a median graph if,
for all distinct $u,v,w\in V(G)$, there is a unique vertex that belongs to
shortest paths between each pair of $u, v$ and $w$. Equivalently, $G$ is a
median graph if $|I_G(u,v)\cap I_G(u,w)\cap I_G(v,w)| =1$ for every triple
$u$, $v$ and $w$ of its vertices \cite{Mulder:78, SM:99}.  We denote the
unique median of three vertices $u$, $v$ and $w$ in a median graph $G$ by
$\med_G(u,v,w)$.  A well-known example of median graphs are trees. In
particular, if we consider rooted trees $T$, then $\lca_T(x,y)$ lies on all
three paths between $x$ and $y$, between $\rho_T$ and $x$ as well as
between $\rho_T$ and $y$. Taking the latter two arguments together, we
obtain the following
\begin{fact}
  If $T$ is a rooted tree, then $\lca_T(x,y) = \med_T(x,y,\rho_T)$ for
  every $x,y \in V(T)$.
\end{fact}

A further example of median graphs are particular grid graphs for which a
planar drawing is provided.  As a direct consequence of Lemma 6 together
with Theorem 7 in \cite{KS:00}, we obtain
\begin{theorem}[{\cite{KS:00}}]
  A connected grid graph $G$ is a median graph if and only if all inner
  faces of $G$ are squares.
  \label{thm:grid-median}
\end{theorem}
If $e$ is an edge in a connected graph $G$ whose removal makes $G$
disconnected, then $G$ is a median graph if and only if both components of
$G-e$ (the graph obtained from $G$ by removing the edge $e$) are median
graphs \cite{KS:00}. Therefore, we obtain the following simple result that
we need for later reference.
\begin{lemma}
  Let $G_1=(V_1,E_1)$ and $G_2=(V_2,E_2)$ be two vertex-disjoint graphs and
  let $v\in V_1$ and $w\in V_2$.  Then,
  $G=(V_1\cup V_2, E_1\cup E_2\cup \{\{v,w\}\})$ is a median graph if and
  only if $G_1$ and $G_2$ are median graphs.  In particular, if $L$ is the
  set of leaves in $G$, then $G$ is median graph if and only if $G-L$ is
  median graph.
  \label{lem:leaf-append}
\end{lemma}

\begin{definition}
  Let $G=(V,E)$ be a rooted median graph with root $\rho$ and a specified
  set $L\subseteq V$ of vertices.  Then,
  \begin{equation*}
    V_{\med_G,L} \coloneqq \{\med_G(\rho,x,y) \mid x,y\in L \text{ and }
    x,y,\rho \text{ pairwise distinct} \}
  \end{equation*}
  denotes the set of all (unique) medians $\med_G(\rho,x,y)$ for all
  distinct pairs of vertices $x,y\in L$.
  \label{def:VL}
\end{definition} 

\begin{definition}
  Let $\delta\colon \Xirr \to \Upsilon$ be a symmetric map. A rooted median
  graph $G = (V,E)$ with root $\rho$ and specified set $L = X\subseteq V$
  \emph{explains} $\delta$ if there is a labeling
  $t\colon V_{\med_G,L} \to \Upsilon$ such that
  $\delta(x,y)=t(\med_G(\rho,x,y))$ for all distinct $x,y\in X$.
  \label{def:explain}
\end{definition}

In the following the set $L$ in Def.\ \ref{def:VL} and \ref{def:explain}
will coincide the leaf set of the median graphs under consideration.

\section{Every map $\delta$ can be explained by a labeled median graph}
\label{sect:general}

We are interested in rooted median graphs that can explain a given map
$\delta$. As noted in the introduction, only a very restricted subclass of
maps, namely the symbolic ultrametrics \cite{Boecker:98}, can be explained
by rooted trees. This is not surprising since $O(|X|^2)$ values of
$\delta(x,y)$ must be explained by only $O(|X|)$ labels at the internal
vertices of a tree. This suggests that rooted median graphs with a large
enough number of inner vertices should be able to explain any given
map. Here we show that this is indeed the case.

\subsection{Extended Hypercubes}

One of the best known median graphs is possibly the hypercube. In
particular, every median graph is a distance-preserving subgraph of some
hypercube \cite{KS:00}.

\begin{definition}
  Let $Q_{n}=(V,E)$ be a hypercube.  An \emph{extended} hypercube
  $\Qext_{n}$ with leaf set $L=\{x_1,\dots,x_n\}$ is obtained from
  $Q_{n}$ as follows: every $x_i\in L$ is leaf-appended to the vertex for
  which only the $i$-th coordinate is a $1$.  The root $\rho$ of
  $\Qext_{n}$ will always be the vertex with coordinate vector
  $(1,\dots,1)$.
\end{definition}

\begin{lemma}	
  $\Qext_{n}$ is a median graph and
  $\med_{\Qext_{n}}(\rho,x_i,x_j) \neq \med_{\Qext_{n}}(\rho,x_k,x_l)$
  whenever $x_i$ and $x_j$, resp.\, $x_k$ and $x_l$ are distinct vertices
  and $|\{i,j\} \cap \{k,l\}|\leq 1$.
  \label{lem:QndistinctMedians}
\end{lemma}
\begin{proof}
  Since $Q_{n}$ as well as all single vertex graphs
  $G_j=(\{x_j\},\emptyset)$, $1\leq j\leq n$ are median graphs, Lemma
  \ref{lem:leaf-append} implies that $\Qext_{n}$ is a median graph.
	
  Now, let $L$ be the set of leaves in $\Qext_{n}$ and let $x_i,x_j\in L$
  and $x_k,x_l\in L$ be two distinct vertices, respectively, such that
  $\{i,j\} \neq \{k,l\}$.  Since $\Qext_{n}$ is a median graph, the medians
  in $\Qext_{n}$ are unique for all three vertices in $\Qext_{n}$ and it is
  easy to verify that
  $\med_{Q_{n}}(\rho,v_i,v_j) = \med_{\Qext_{n}}(\rho,v_i,v_j)$ for every
  distinct $v_i, v_j \in V(Q_{n})$ with $v_i, v_j \ne \rho$.  By
  construction, $x_{\ell}\in L$ is leaf-appended to the unique vertex
  $v_{\ell}$ with all coordinates being $0$ except the $\ell$-th coordinate
  which is a $1$ and every shortest path from $x_{\ell}$ to any other
  vertex must contain $v_{\ell}$. The latter two arguments imply that it
  suffices to show that
  $\med_{Q_{n}}(\rho,v_i,v_j) \neq \med_{Q_n}(\rho,v_k,v_l)$ in order to
  show that
  $\med_{\Qext_{n}}(\rho,x_i,x_j) \neq \med_{\Qext_{n}}(\rho,x_k,x_l)$.

  Assume for contradiction that
  $w\coloneqq \med_{Q_{n}}(\rho,v_i,v_j) = \med_{Q_{n}}(\rho,v_k,v_l)$.
  Based on Lemma \ref{lem:propertyCart} and the respective Hamming
  distances we easily obtain,
  \begin{equation*}
    d_{Q_{n}}(v_i,w)+d_{Q_{n}}(w,v_j)= d_{Q_{n}}(v_i,v_j) = 2
    =d_{Q_{n}}(v_k,v_l) = d_{Q_{n}}(v_k,w)+d_{Q_{n}}(w,v_l).
  \end{equation*}
  There are two possibilities for $w$ to satisfy
  $d_{Q_{n}}(v_i,w)+d_{Q_{n}}(w,v_j)= d_{Q_{n}}(v_i,v_j) = 2$, that is, $w$
  either the vertex with coordinate vector $(0,\dots,0)$ or the vertex for
  which the coordinate vectors have precisely two $1$s, namely the $i$-th
  and $j$-th coordinate.  Analogously, for
  $d_{Q_{n}}(v_k,v_l) = d_{Q_{n}}(v_k,w)+d(w,v_l)=2$, the median $w$ is
  either the vertex $(0,\dots,0)$ or the vertex for which the coordinate
  vectors have precisely two $1$s, namely the $k$-th and $l$-th coordinate.
  This together with $\{i,j\} \neq \{k,l\}$ implies that $w$ must be the
  vertex $(0,\dots,0)$.  Since $w =\med_{Q_{n}}(\rho,v_i,v_j)$ it must also
  hold that $d_{Q_{n}}(v_j,w)+d_{Q_{n}}(w,\rho)=d_{Q_{n}}(v_j,\rho) =
  n-1$. However, since $d_{Q_{n}}(v_j,w) =1$ and $d_{Q_{n}}(w,\rho) = n$,
  we obtain $d_{Q_{n}}(v_j,w)+d_{Q_{n}}(w,\rho)=d_{Q_{n}}(v_j,\rho) = n+1$;
  a contradiction.
\end{proof}

\begin{corollary}
Every  symmetric map $\delta\colon \Xirr \to \Upsilon$
can be explained by a labeled extended hypercube $(\Qext_{|X|},t)$.
\label{cor:Qext-explains}
\end{corollary}
\begin{proof}
	Let $X$ be the set of leaves in $\Qext_{|X|}$. 
  Put $t(\med_{\Qext_{|X|}}(\rho,x,y)) = \delta(x,y)$ for all distinct
  $x,y\in X$.  By Lemma \ref{lem:QndistinctMedians}, we have
  $\med_{\Qext_{|X|}}(\rho,x_i,x_j) \neq \med_{\Qext_{|X|}}(\rho,x_k,x_l)$
  given that $\{i,j\} \neq \{k,l\}$ and $i\neq j$ and $k\neq l$.  Hence,
  $t(\med_{\Qext_{|X|}}(\rho,x,y))$ is well-defined and $(\Qext_{|X|} , t)$
  explains $\delta$.
\end{proof}

An illustrative example for Cor.\ \ref{cor:Qext-explains} is provided in
Fig.\ \ref{fig:exmpl}.  Of course, $\Qext_{|X|}$ has $O(2^{|X|})$ vertices
and many vertices are not part of $V_{\med_{\Qext_{|X|}}, X}$ which makes
this construction intractable in practice.  Hence, we focus now on graphs
that have significantly less vertices and still explain a symmetric map
$\delta$.

\subsection{Extended Halfgrids}

We next consider a class of much smaller median graphs.
\begin{definition} 
  A \emph{half-grid} $H_{n}$ with $n\geq 2$ is defined as
  $H_2=P_2 \Box P_2$ or, if $n\geq 3$, then $H_n$ is obtained from the
  Cartesian product $P_n \Box P_n$ by removing all vertices with coordinate
  vectors $(i,j)$ with $1\leq j\leq n-2$ and $j+2\leq i\leq n$.  and its
  incident edges, see also Fig.\ \ref{fig:halfG-exmpl-full}.
\end{definition}

\begin{fact}
  By construction, $H_{n}$ is a subgraph of $P_n \Box P_n$ that is induced
  by the vertices with coordinate vectors $(i,j)$ such that either $i=1$
  and $1\leq j\leq n$ or $2\leq i\leq n$ and $i-1\leq j\leq n$.

  Analogously, $H_{n}$ is a subgraph of $P_n \Box P_n$ that is induced by
  the vertices with coordinate vectors $(i,j)$ such that either $j=n$ and
  $1\leq i\leq n$ or $1\leq j\leq n-1$ and $1\leq i\leq j+1$.
  \label{obs:coord}
\end{fact}

\begin{figure}[t]
  \begin{center}
  	% in ...NEW.pdf waren ein paar Klammern abgeschnitten
    \includegraphics[width=0.85\textwidth]{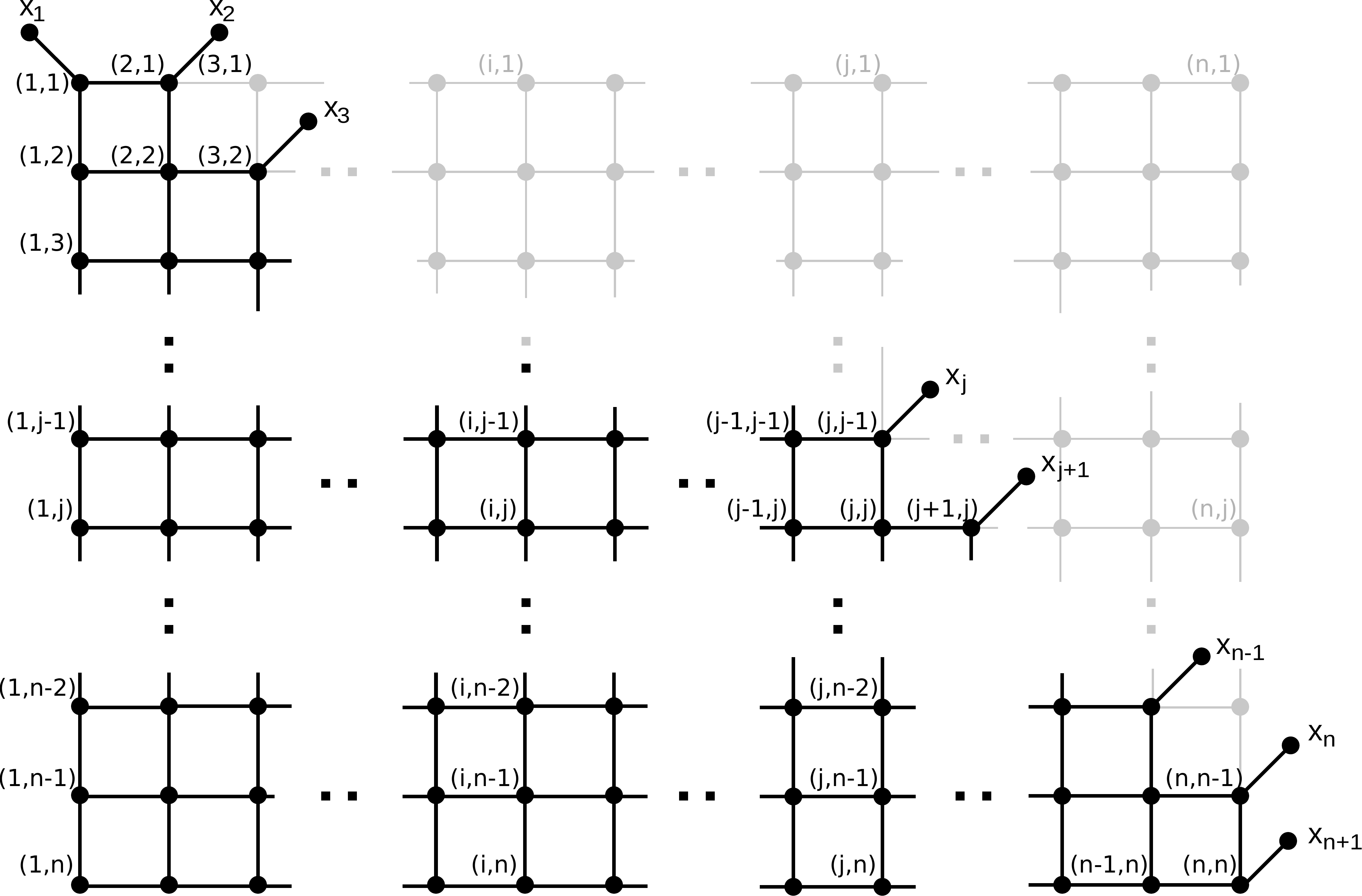}
  \end{center}
  \caption{Sketch of the extended half-grid $\Hext_{n+1}$ on
    $X=\{x_1,\dots,x_{n+1}\}$ that is obtained from $P_n\Box P_n$ by
    removal of all gray-colored edges and vertices.  For better
    readability, we have drawn only a few vertex coordinate vectors that
    always belong to the vertex at the lower-right of the respective
    coordinate vector.  }
  \label{fig:halfG-exmpl-full}
\end{figure}

\begin{figure}[t]
  \begin{center}
    \includegraphics[width=0.7\textwidth]{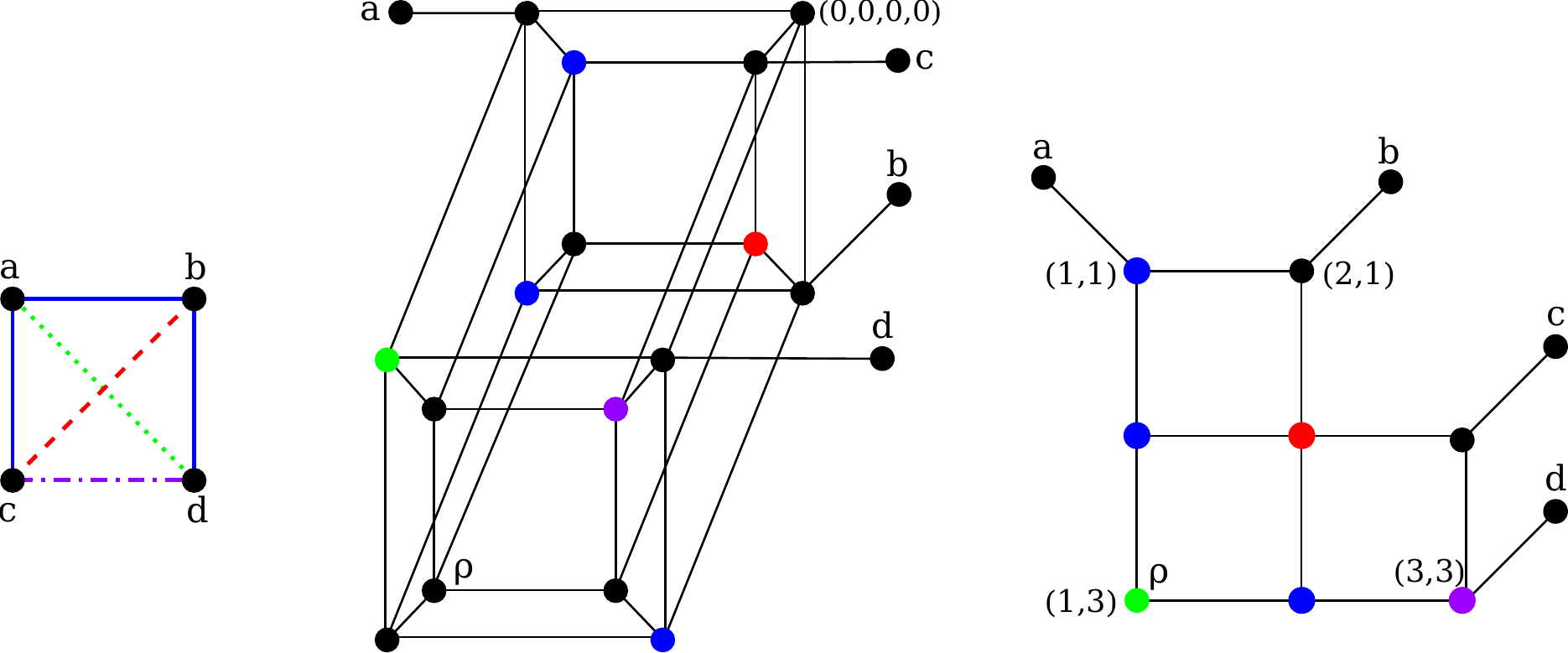}
  \end{center}
  \caption{Left: the graph-representation of a map
    $\delta\colon\Xirr\to\Upsilon$ with $X=\{a,b,c,d\}$ and
    $\delta(a,b)=\delta(a,c)=\delta(b,d)=\mathrm{blue}$ \emph{(solid line)},
    $\delta(b,c)=\mathrm{red}$ \emph{(dashed line)},
    $\delta(c,d)=\mathrm{purple}$ \emph{(dashed-dotted line)},
    $\delta(a,d)=\mathrm{green}$ \emph{(dotted line)}.  For the two graphs
    $(\Qext_{4},t)$ (middle) and $(\Hext_{4},t)$ (right) all non-black
    vertices belong to $V_{\med_{\Qext_{4}},X}$ and
    $V_{\med_{\Hext_{4}},X}$, respectively. Both labeled graphs
    $(\Qext_{4},t)$ and $(\Hext_{4},t)$ explain $\delta$.  }
  \label{fig:exmpl}
\end{figure}

\begin{lemma}
  $H_{n}$ is a median graph.
\label{lem:hg-median}
\end{lemma}
\begin{proof}
  By construction $H_n$ is obtained from $G=P_n\Box P_n$ be removal of some
  vertices and its incident edges. All removed vertices are of the form
  $(i,j)$ with $1\leq j\leq n-2$ and $j+2\leq i\leq n$. We continue to show
  that each inner face of $H_n$ is a square.

  Consider a planar embedding of $P_n\Box P_n$ which can be chosen in a way
  such that all inner faces are squares.  Let us first remove vertex
  $v=(n,1)$.  This yields one square deletion and does not produce any new
  inner face in the graph $G' = G-\{v\}$. Now, we remove vertex $w=(n-1,1)$
  from $G'$ and again, one square is removed and no new inner face has been
  created. This step can repeated until we end in some graph where all
  vertices $(n,1),\dots,(3,1)$ have been removed. In this graph, all inner
  faces are still squares and thus, by Thm.\ \ref{thm:grid-median}, it is a
  median graph.  Now we proceed for all $j$ from $2$ to $n-2$ and delete
  stepwisely all vertices $(n,j),\dots, (j+2,j)$. In each step exactly one
  square is removed and no new inner face is created.

  Hence, $H_n$ is a connected planar grid graph where all inner faces are
  squares.  By Thm.\ \ref{thm:grid-median}, $H_n$ is a median graph.
\end{proof}

\begin{definition}
  Let $H_{n}=(V,E)$ be a half-grid, $n\geq 2$.  An \emph{extended}
  half-grid $\Hext_{n+1}$ with leaf set $L=\{x_1,\dots,x_{n+1}\}$ is
  obtained from $H_{n}$ as follows: Make $x_1$ leaf-appended to vertex
  $(1,1)$, $x_{n+1}$ leaf-appended to vertex $(n,n)$ and $x_j$
  leaf-appended to vertex $(j,j-1)$, $2\leq j\leq n$ in $H_{n}$.

  The root $\rho$ of $\Hext_{n+1}$ will always be the vertex with
  coordinate vector $(1,n)$.
\end{definition}

\begin{proposition}
  Let $\Hext_{n+1}$ be an extended half-grid with leaf set
  $L=\{x_1,\dots,x_{n+1}\}$. Then, $\Hext_{n+1}$ is a median graph and for
  all leaves $x_i$ and $x_j$, $i<j$ we have
  $\med_{\Hext_{n+1}}(\rho,x_i,x_j)=(i,j-1)$.
\label{prop:ext-hg-median}
\end{proposition}
\begin{proof}
  For simplicity we put $ \med(\dots) =\med_{\Hext_{n+1}}(\dots)$.  Since
  $H_{n}$ as well as all single vertex graphs $G_j=(\{x_j\},\emptyset)$,
  $1\leq j\leq n+1$ are median graphs, we can apply Lemma
  \ref{lem:hg-median} together with Lemma \ref{lem:leaf-append} to conclude
  that $\Hext_{n+1}$ is a median graph.

  We continue by showing that $\med(\rho,x_i,x_j)=(i,j-1)$.  We have
  leaf-appended each $x_j\in X$ to a unique vertex $v_j$ in $H_n$, i.e.,
  $v_1 = (1,1)$ $v_{n+1} = (n,n)$ and $v_j = (j,j-1)$, $2\leq j\leq n$.
  For simplicity, we put $H\coloneqq \Hext_{n+1}$ and $G = P_n \Box P_n$.

  In the following, we will make frequent use of the following observation:
  Since $H_n\subseteq G$, we have $d_H(x,y)= d_{H_n}(x,y)\geq d_{G}(x,y)$
  for all $x,y\in V(H_n)\subseteq V(H)$.  Hence, if there is a path
  $P_H(x,y)$ of length $d_{G}(x,y)$, then $P_H(x,y)$ must be a shortest
  path in $H$ between $x$ and $y$. For simplicity we denote with ${P}_{ij}$
  the path $P_H(v_i,v_j)$ in $H$ connecting $v_i$ and $v_j$, $i<j$.

  \begin{claim}
    \label{c:aans}
    There is a shortest path $P_{i,j}$ in $H$ connecting $v_i$ and $v_j$,
    $i<j$ that contains the vertex $w_{i,j} = (i,j-1)$.
  \end{claim}
  \begin{claim-proof}
    Let us start with the special case $v_1 = (1,1)$ and $v_{n+1} = (n,n)$.
    By construction of $H_n$, all vertices $(1,k)$ with $1\leq k\leq n$ and
    all vertices $(k,n)$ with $1\leq k \leq n$ are contained in $H_n$.
    Thus, there is a path $P_{1,n+1}$ induced by the vertices
    $(1,1), \dots, (1,n), \dots,(n,n)$ that has length $2(n-1)=2n-2$.  By
    Lemma \ref{lem:propertyCart}, $d_{G}(v_1,v_{n+1}) = (n-1)+(n-1)=2n-2$,
    which implies that $P_{1,n+1}$ is a shortest path in $H$ that, in
    particular, includes $w_{1,n+1} = (1,n)$.

    Now, consider $v_1 = (1,1)$ and $v_j = (j,j-1)$, $2\leq j\leq n$.
    Note, by construction of $H_n$, all vertices $(1,k)$ with
    $1\leq k\leq j-1 < n$ and all vertices $(k,j-1)$ with $1\leq k \leq j$
    are contained in $H_n$.  Thus, there is a path $P_{1,j}$ induced by
    vertices $(1,1), \dots, (1,j-1), \dots,(j,j-1)$ of length
    $(j-2)+(j-1)$, $2\leq j\leq n$.  By Lemma \ref{lem:propertyCart},
    $d_{G}(v_1,v_j) = (j-1)+(j-2)$, which implies that $P_{1,j}$ is a
    shortest path in $H$ that, in particular, includes $w_{1,j} = (1,j-1)$.

    By similar arguments, there is a path $P_{i,n+1}$ from $v_i = (i,i-1)$
    and $v_{n+1} = (n,n)$, $2\leq i\leq n$ along the vertices
    $(i,i-1), \dots, (i,n), \dots,(n,n)$ of length $(n-i+1)+(n-i)$.  By
    Lemma \ref{lem:propertyCart}, $d_{G}(v_i,v_{n+1}) = (n-i)+(n-i+1)$,
    which implies that $P_{i,n+1}$ is a shortest path in $H$ that, in
    particular, includes $w_{i,n+1} = (i,n)$.

    Now consider $v_i = (i,i-1)$ and $v_j = (j,j-1)$, $2\leq i<j\leq n$.
    By Remark \ref{obs:coord}, all vertices $(i,k)$ with $i-1\leq k \leq n$
    and all vertices $(k,j-1)$ with $i-1\leq k\leq (j-1)+1=j$ (since
    $j-1\neq n$) are contained in $H_n$.  Thus, there is a path $P_{i,j}$
    induced by vertices
    $(i,i-1), (i,i), \dots, (i,j-1), (i+1,j-1),\dots,(j,j-1)$ of length
    $(j-1-(i-1))+ (j-i) = 2j-2i$.  By Lemma \ref{lem:propertyCart},
    $d_{G}(v_i,v_j) = (j-i) + (j-1-(i-1)) = 2j-2i$, which implies that
    $P_{i,j}$ is a shortest path in $H$ that, in particular, includes
    $w_{i,j} = (i,j-1)$.
  % This proves Claim 1.
  \end{claim-proof}

  \begin{claim}
    \label{c:zwaa}
    There are shortest paths ${P}_H(\rho,v_i)$ and ${P}_H(\rho,v_j)$ that
    both contain vertex $w_{i,j} = (i,j-1)$ for all $1\leq i < j\leq n+1$.
  \end{claim}
  \begin{claim-proof}
    For simplicity, we put ${P}_{\rho,i}\coloneqq {P}_H(\rho,v_i)$ and
    ${P}_{\rho,j}\coloneqq {P}_H(\rho,v_j)$.

    Let us start again with $v_1 = (1,1)$. In this case, there is a path
    ${P}_{\rho,1}$ from $\rho$ to $v_1$ along the vertices
    $\rho=(1,n), \dots, (1,1)=v_1$ of length $n-1$ which is, by
    Lemma \ref{lem:propertyCart}, the same as $d_{G}(\rho,v_1)$. Hence,
    ${P}_{\rho,1}$ is a shortest path in $G$.  In particular,
    $w_{1,j}=(1,j-1)$ is contained in ${P}_{\rho,1}$ for all $1<j\leq n+1$.

    By similar arguments and using the path ${P}_{\rho, n+1}$ along the
    vertices $\rho=(1,n), \dots, (n,n)=v_{n+1}$, the path ${P}_{\rho, n+1}$
    contains all vertices $w_{i,n+1}=(i,n)$ with $1\leq i<j\leq n+1$.
     
    Now assume that $v_i$ and $v_j$ are chosen such that $1<i<j<n+1$.  By
    similar arguments as in the proof of Claim~\ref{c:aans}, there is a
    path ${P}_{\rho,i}$ along the vertices
    $\rho=(1,n), \dots, (i,n), \dots, (i,j-1), \dots, (i,i-1) = v_i$.  This
    path ${P}_{\rho, i}$ has length $(i-1)+(n-(i-1)) =n$, Moreover, there
    is a path ${P}_{\rho,j}$ along the vertices
    $\rho=(1,n), \dots,(1,j-1), \dots, (i,j-1), \dots, (j,j-1) = v_j$.
    This ${P}_{\rho, j}$ has length $(n-(j-1))+(j-1)=n$, which is, by Lemma
    \ref{lem:propertyCart}, the same as $d_{G}(\rho,v_j)$.  For both cases,
    Lemma \ref{lem:propertyCart} implies that ${P}_{\rho, i}$ and
    ${P}_{\rho, j}$ are shortest paths.  In particular, both paths contain
    $w_{i,j}=(i,j-1)$.
    % This proves Claim 2.
  \end{claim-proof}

  We are now in the position to prove the final statement
  $\med(\rho,x_i,x_j)=(i,j-1)$ where $x_i,x_j\in X$, $i<j$.  Since each
  $x_i$ is leaf-appended to $v_i$ every shortest path from $x_i$ to every
  other vertex must contain $v_i$.  In other words, every shortest path
  from $x_i$ to some vertex $z$ consists of the edge $\{x_i,v_i\}$ and a
  shortest path from $v_i$ to $z$.  Thus, Claims~\ref{c:aans}
  and~\ref{c:zwaa} imply that $w_{ij}=(i,j-1)$ is contained in a shortest
  path from $x_i$ to $x_j$ as well as in a shortest path from $\rho=(1,n)$
  to $x_i$ and $x_j$, respectively.  Finally, since $H$ is a median graph
  it must hold that $\med(\rho,x_i,x_j)=(i,j-1)$, $1\leq i < j \leq n+1$.
\end{proof}

\begin{corollary}
  Let $\Hext_{n+1}$ be an extended half-grid with leaf set
  $L=\{x_1,\dots,x_{n+1}\}$.  Then, for each
  $w\in V_{\med_{\Hext_{n+1}},L}$ there is a unique pair $(x_i,x_j)$, $i<j$
  such that $w = \med_{\Hext_{n+1}}(\rho,x_i,x_j)$.
  \label{cor:unique-med}
\end{corollary}

\begin{theorem}
  For every symmetric map $\delta\colon \Xirr \to \Upsilon$, there is a
  labeled median graph $(G,t)$ with $O(|X^2|)$ vertices and leaf-set $X$
  that explains $\delta$ and such that its root $\rho_G$ is in
  $V_{\med_G,X}$.
  \label{thm:median-explains}
\end{theorem}
\begin{proof}
  If $|X|=2$, then one can easily verify that this can explained by a
  rooted tree with one inner vertex and two leaves.  Consider $\Hext_{|X|}$
  with leaf set $X=\{x_1, \dots, x_{n+1}\}$, $|X|\geq 3$ and root $\rho$.
  By construction, $\Hext_{|X|}$ has $O(|X^2|)$ vertices.  Put
  $t(\med_{\Hext_{|X|}}(\rho,x,y)) = \delta(x,y)$ for all distinct
  $x,y\in X$.  By Cor. \ref{cor:unique-med},
  $\med_{\Hext_{|X|}}(\rho,x,y))$ is uniquely determined for all distinct
  $x,y\in X$. Hence, $t(\med_{\Hext_{|X|}}(\rho,x,y))$ is well-defined and
  $(G=\Hext_{|X|} , t)$ explains $\delta$. Moreover, by Prop.\
  \ref{prop:ext-hg-median},
  $\med_{\Hext_{|X|}}(\rho,x_1,x_{n+1})=(1,n) = \rho$. Therefore,
  $\rho\in V_{\med_{G},X}$.
\end{proof}
An example for the construction as in the proof of Thm.\
\ref{thm:median-explains} is provided in Fig.\ \ref{fig:exmpl}.

While $\Qext_{|X|}$ has $2^{|X|}+|X|$ vertices, the graph $\Hext_{|X|}$ has
only $\Theta(|X|^2)$ vertices, and thus, is more space-efficient. There are
maps for which we cannot avoid that the graph that explains it has
$\Theta(|X|^2)$ vertices. In particular, if
$\delta\colon \Xirr \to \Upsilon$ is a surjective map with
$|\Upsilon| = \vert \binom{X}{2}\vert$, then
$\delta(x,y)\neq \delta(x',y')$ for all pairs $(x,y), (x',y') \in \Xirr$
with $\{x,y\}\neq \{x',y'\}$. In this case, all the labels of the
respective medians must be distinct and thus, $\Theta(|X|^2)$ medians must
exist.  In other words, halfgrids are in some sense optimal if all
$\delta(x,y)$ are distinct.  In general, however, we want to explain maps
$\delta$ by median graphs that are closer to trees. This leads us directly
to the concept of the modular decomposition of a map which is explained in
the next section.
 	
\section{Median graphs from Modular Decomposition Trees}
\label{sect:MDT}

This section makes extensive use of results established in \cite{HSW:16}
for so-called 2-structures (cf.\
\cite{ER1:90,ER2:90,Ehrenfeucht1995,ehrenfeucht1999theory,EHPR:96,McConnell:95,
  EHMS:94,ehrenfeucht1999theory,EHRENFEUCHT1990343,EHR:94} and (not
necessarily symmetric) maps $\delta$. A (labeled) \emph{2-structure} is a
triple $g = (X,\Upsilon,\delta)$ where $X$ and $\Upsilon$ are nonempty sets
and $\delta\colon \Xirr \to \Upsilon$ is a map. Since 2-structures are
essentially determined by $\delta\colon \Xirr\to\Upsilon$ we use such maps,
instead of 2-structures, which is more suitable for our purposes. The idea
underlying this section is to start from the modular decomposition tree of
a symmetric map $\delta$ and to ``expand'' this tree in a principled manner
into a rooted median graph that explains $\delta$. We thus start with the
notion of modules for symmetric maps.
\begin{definition}
  A \emph{module} of a symmetric map $\delta\colon \Xirr\to\Upsilon$ is a
  subset $M\subseteq X$ such that $\delta(x,z)=\delta(y,z)$ holds for all
  $x,y \in M$ and $z\in X\setminus M$. A module $M$ of $\delta$ is
  \emph{strong} if $M$ does not overlap with any other module of $\delta$,
  that is, $M\cap M' \in \{M, M', \emptyset\}$ for all modules $M'$ of
  $\delta$.
\end{definition}
We write $\M(\delta)$ for the set of all modules of a symmetric map
$\delta$ and $\Ms(\delta)\subseteq \M(\delta)$ for the set of all strong
modules of $\delta$. The empty set $\emptyset$, the complete vertex set
$X$, and the singletons $\{v\}$ are always modules. They are called the
\emph{trivial} modules of $\delta$. We will assume from here on, that a
module is non-empty unless otherwise indicated.

The set $\Ms(\delta)$ of strong modules is uniquely determined
\cite{HSW:16,EHMS:94}.  While there may be exponentially many modules, the
size of the set of strong modules is in $O(|X|)$ \cite{EHMS:94}. In
particular, $X$ and the singletons $\{v\}$, $v\in X$ are strong
modules. Since strong modules do not overlap this implies that
$\Ms(\delta)$ forms a hierarchy and, by Prop.\ \ref{A:prop:hierarchy},
gives rise to a unique tree representation $T_{\delta}$ of $\delta$, known
as the \emph{modular decomposition tree} (\emph{MDT}) of $\delta$. The
vertices of $T_\delta$ are (identified with) the elements of $\Ms(\delta)$.
Adjacency in $T_\delta$ is defined by the maximal proper inclusion
relation, that is, there is an edge $\{M,M'\}$ between
$M,M'\in \Ms(\delta)$ iff $M\subsetneq M'$ and there is no
$M''\in \Ms(\delta)$ such that $M\subsetneq M'' \subsetneq M'$.  The root
of $T_\delta$ is (identified with) $X$ and every leaf $v$ corresponds to the
singleton $\{v\}$, $v\in X$.

Uniqueness and the hierarchical structure of $\Ms(\delta)$ implies that
there is a unique partition $\Mmax(\delta) = \{M_1,\dots, M_k\}$ of $X$
into maximal (w.r.t.\ inclusion) strong modules $M_j\ne X$ of $\delta$
\cite{ER1:90,ER2:90}.  Since $X\notin \Mmax(\delta)$ the set
$\Mmax(\delta)$ consists of $k\ge 2$ strong modules, whenever $|X|>1$.

For later reference, we recall
\begin{lemma}[\cite{ER1:90}, Lemma 4.11]
  Let $M_1, M_2 \in \M(\delta)$ be two disjoint modules of a symmetric map
  $\delta$.  Then there is a unique label $i\in \Upsilon$ such that
  $\delta(x,y) = \delta(y,x) = i$ for all $x \in M_1$ and $y \in M_2$.
  \label{lem:arcs-modules}
\end{lemma}
In order to infer $\delta$ from $T_\delta$ we need to determine the label 
$\delta(x,y)$ of all pairs of distinct leaves $x,y$ from $T_\delta$.  Hence,
we need to define a labeling function $t_\delta$ that assigns this ``missing
information'' to the inner vertices of $T_\delta$. 

The simplest case for the construction of $t_\delta$ is given by symmetric
maps $\delta\colon \Xirr\to\Upsilon$ that satisfy the following two
axioms:				
\begin{itemize}
\item[(U1)] there exists no subset $\{x,y,u,v\}\in \binom{X}{4}$ such that
  $\delta(x,y)=\delta(y,u)=\delta(u,v) \neq
  \delta(y,v)=\delta(x,v)=\delta(x,u)$.
\item[(U2)] $|\{\delta(x,y), \delta(x,z),\delta(y,z)\}| \le 2$ for all
  $x,y,z\in X$.
\end{itemize}
A symmetric map that satisfies (U1) and (U2) is called a \emph{symbolic
  ultrametric} \cite{Boecker:98}. In this case, there is a unique vertex
labeled tree $(T^*,t^*)$, called \emph{discriminating symbolic
  representation} of $\delta$, that satisfies $t^*(v)\neq t^*(u)$ for all
edges $\{u,v\}\in E(T^*)$ with $u,v \in V^0(T^*)$ being inner vertices and
for which $t(\lca_T(x,y)) = \delta(x,y) $ for all distinct $x,y\in X$ (cf.\
\cite{Boecker:98} and \cite[Prop.\ 1]{Hellmuth:13a}).  Hence, given the
discriminating symbolic representation $(T^*,t^*)$ of a symbolic
ultrametric $\delta$ we can uniquely recover $\delta$ from $(T^*,t^*)$.

Now consider the MDT $T_\delta$ of $\delta$.  In case $\delta$ is a
symbolic ultrametric we can also equip $T_\delta$ with a labeling
$t_\delta$, by setting $t_\delta(\lca_{T_{\delta}}(x,y))= \delta(x,y)$ for
all distinct $x,y\in X$.  If $\delta$ is a symbolic ultrametric then Lemma
7 and Theorem 6 and 7 in \cite{HSW:16} imply that $t_\delta$ is
well-defined and satisfies $t_\delta(v)\neq t_\delta(u)$ for all edges
$\{u,v\}\in E(T_\delta)$ with $u,v \in V^0(T_\delta)$.  In particular, we
have $\delta(x,y) = i$ if and only if $t_\delta(\lca_{T_{\delta}}(x,y))=i$.
Since the discriminating symbolic representation $(T^*,t^*)$ and modular
decomposition trees $(T_\delta,t_\delta)$ are unique (up to isomorphism),
we can conclude that the two trees $(T^*,t^*)$ and $(T_\delta,t_\delta)$
must be isomorphic.  We summarize this discussion in the following
\begin{theorem}
  \label{thm:symbolic}
  Suppose $\delta\colon \Xirr\to\Upsilon$ is a symmetric map.  Then there
  is a discriminating symbolic representation $(T^*,t^*)$ of $\delta$ if
  and only if $\delta$ is a symbolic ultrametric. In this case, $(T^*,t^*)$
  and the labeled MDT $(T_\delta,t_\delta)$ are isomorphic.
\end{theorem}
Note, for symbolic ultrametrics $\delta$, we thus obtain a vertex labeled
median graph $(T_\delta, t_\delta)$ that explains $\delta$, since
$\lca_{T_\delta}(x,y) = \med_{T_\delta}(\rho, x,y)$.

\begin{figure}[t]
  \begin{center}
    \includegraphics[width=0.3\textwidth]{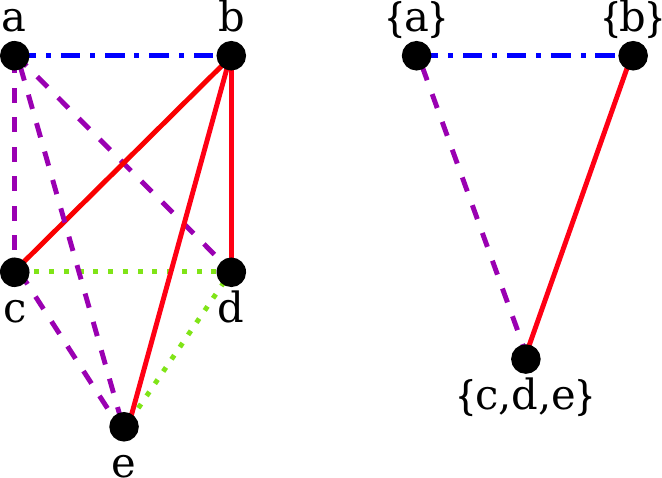}
  \end{center}
  \caption{Left: The graph-representation of a map
    $\delta\colon\Xirr\to\Upsilon$ with $X=\{a,b,c,d,e\}$ and
    $\delta(a,b)=\mathrm{blue}$ \emph{(dashed-dotted-line)},
    $\delta(b,c)=\delta(b,d)=\delta(b,e)=\mathrm{red}$ \emph{(solid-line)},
    $\delta(a,c)=\delta(a,d)=\delta(a,e)=\delta(c,e)=\mathrm{purple}$
    \emph{(dashed-line)}, $\delta(c,d)=\delta(d,e)=\mathrm{green}$
    \emph{(dotted-line)}.  Right: The graph-representation of the map
    $\delta/\Mmax(\delta) \colon \Mmax(\delta)^{2}_\mathrm{irr} \to
    \Upsilon$.  Note, $\delta$ is not a symbolic ultrametric, since it does
    not satisfy (U2) as e.g.\
    $|\{\delta(a,b), \delta(a,c),\delta(b,c)\}| =3$.  The strong modules of
    $\delta$ are the trivial modules $X$ and $\{x\}$ for all $x\in X$ as
    well as $\{c,e\}$ and $M\coloneqq \{c,d,e\}$. Therefore,
    $\Mmax(\delta) = \{\{a\},\{b\}, M\}$ and for the quotient map
    $\delta/\Mmax(\delta)$ we have
    $\delta/\Mmax(\delta)(\{a\},\{b\}) =\mathrm{blue}$,
    $\delta/\Mmax(\delta)(\{a\},M) =\mathrm{purple}$ and
    $\delta/\Mmax(\delta)(\{b\},M) =\mathrm{red}$. Note, in this example,
    $\delta/\Mmax(\delta)$ is prime since $\M(\delta/\Mmax(\delta))$
    consists of the trivial modules $\{\{a\}\}$, $\{\{b\}\}$ and $\{M\}$
    only.  }
  \label{fig:quotient}
\end{figure}

Not all maps $\delta$ are symbolic ultrametrics. The modular decomposition
tree $T_{\delta}$ still exists, but in general there will be no labeling
$t_\delta$ such that $t_\delta(\lca_{T_{\delta}}(x,y)) = \delta(x,y)$ holds
for all $x\ne y$, see Fig.\ \ref{fig:MDT}.
As a remedy, let us consider the following
\begin{definition}
  Let $\delta\colon\Xirr\to\Upsilon$ be a symmetric map.  A \emph{quotient
    map} is the map
  $\delta/\Mmax(\delta) \colon \Mmax(\delta)^{2}_\mathrm{irr} \to \Upsilon$
  obtained from $\delta$ by putting, for all $M,M'\in \Mmax(\delta)$,
  $\delta/\Mmax(\delta)(M,M') = \delta(x,y)$ for some $x\in M$, $y\in M'$.
\end{definition}
We first note that $\delta/\Mmax(\delta)$ is well-defined because
$\Mmax(\delta)$ is well-defined and for every two distinct and, therefore,
disjoint modules $M, M' \in \Mmax(\delta)$ there is a label $i\in \Upsilon$
with $\delta(x,y)=i$ for all $x\in M$ and $y\in M'$ (cf.\ Lemma
\ref{lem:arcs-modules}).

In addition to symbolic ultrametrics, there are two other important
subclasses of symmetric maps $\delta\colon \Xirr\to\Upsilon$:
\begin{itemize} 
\item[] A map $\delta$ is \emph{prime} if $\M(\delta)$ consists of trivial
  modules only.
\item[] $\delta$ is \emph{complete} if for all $(x,y), (x',y')\in \Xirr$,
  $\delta(x,y) = \delta(x',y')$.
\end{itemize} 
Although maps $\delta$ are not necessarily prime or complete, their
quotients $\delta/\Mmax(\delta)$ are always either of one or the other
type.
\begin{lemma}[\cite{ER2:90,EHPR:96,HSW:16}]
  Let $\delta$ be a symmetric map. Then the quotient $\delta/\Mmax(\delta)$
  is either complete or prime.  If $\delta$ is a symbolic ultrametric, then
  $\delta/\Mmax(\delta)$ is complete.
\end{lemma}
An illustrative example of the notation established above is provided in
Fig.\ \ref{fig:quotient} and \ref{fig:MDT}. We shall say that an
\emph{inner} vertex $v$ of $T_\delta$ (or, equivalently, the module $L(v)$
where $L=L_{T_\delta}$) is complete or prime if the quotient
$\delta_{|L(v)}/\Mmax(\delta_{|L(v)})$ is complete or prime, respectively.
We can now adjust the labeling function by setting
\begin{equation*}
  t_\delta(v) = 
  \begin{cases} 
    \mbox{prime, } &\mbox{if } v \mbox{ is prime}  \\
    i &\mbox{else, in which case } v=\lca_{T_{\delta}}(x,y)
    \text{ and } \delta(x,y)=i
    \mbox{ for some leaves } x,y\in L(v) 		
  \end{cases}
\end{equation*}
for all inner vertices $v\in V^0(T_\delta)$. 

\begin{theorem}{\cite[Thm.\ 3 and Prop.\ 1]{HSW:16}} 
  For a symmetric map $\delta\colon \Xirr\to\Upsilon$ the following
  statements are equivalent:
  \begin{enumerate}
  \item $\delta$ is a symbolic ultrametric.
  \item The labeled MDT $(T_\delta,t_\delta)$ of $\delta$ has no inner
    vertex $v$ labeled prime, that is, the quotient
    $\delta_{|L(v)}/\Mmax(\delta_{|(L(v)})$ is always complete where $L=X$
    denotes the leaf set of $T_\delta$.
  \end{enumerate}
  \label{thm:char-old}
\end{theorem}

In order to infer $\delta$ from $T_\delta$ we need to determine the label
$\delta(x,y)$ of all pairs of distinct leaves $x,y$ of $T_\delta$.  In the
case of prime nodes, however, we must therefore drag the entire information
of the quotient maps. An alternative idea is to replace prime vertices by
suitable median graphs and extend the labeling function $t_\delta$ that
assigns the ``missing information'' to the inner vertex of the new graph.

\begin{definition}[prime-vertex replacement (pvr) graphs] 
\label{def:pvr}
Let $\delta\colon \Xirr\to\Upsilon$ be a symmetric map with MDT
$(T_\delta,t_\delta)$ that has leaf set $L=X$.  Denote by $\mathcal{P}$ be
the set of all prime vertices in $T_\delta$.  A
\emph{prime-vertex replacement} (\emph{pvr}) graph $(G^*, t^*)$ of
$(T_\delta,t_\delta)$ is obtained by the following procedure:
\begin{enumerate} 
\item For all $v\in \mathcal{P}$, remove all edges $\{v,u\}$ with
  $u\in \child_{T_\delta}(v)$ from $T_\delta$ to obtain the forest
  $(T',t_\delta)$. \label{step:T'}
  We note that each child
  $u\in \child_{T_\delta}(v)$ corresponds to a unique module
  $L(u) \in \Mmax(\delta_{|(L(v)})$.
\item For all $v\in \mathcal{P}$  choose a median graph $G_v$ with root $v$
  and leaf-set $L(G_v) = \{u\mid L(u)\in \Mmax(\delta_{|(L(v)})\}$  and
  labeling $t_v\colon V_{\med_{G_v},L(G_v)}\to \Upsilon$ such that
  $(G_v,t_v)$ explains $\delta_{|L(v)}/\Mmax(\delta_{|(L(v)})$ and such
  that $v\in V_{\med_{G_v},L(G_v)}$.
  \label{step:Gv}
\item For all $v\in \mathcal{P}$, add $G_v$ to $T'$ by identifying the root
  of $G_v$ with $v$ in $T'$ and each leaf $u$ of $G_v$ with the
  corresponding child $u\in \child_{T_\delta}(v)$, for all $v\in P$.  This
  results in a pvr graph $G^*$.
\item \label{step:color} Let
  $W(G^*) = V^0(T_\delta)\cup \bigcup_{v\in \mathcal P}
  V_{\med_{G_v},L(G_v)} $ be the set of vertices that either obtained a
  label in $T_\delta$ or in one of the chosen median graphs $(G_v,t_v)$.
  We define a new labeling
  $t^*\colon W(G^*)\to \Upsilon$ by putting, for all $v\in W(G^*)$,
  \begin{equation*}
    t^*(v) =
    \begin{cases} 
      t_\delta(v) &\mbox{if } v\in V^0(T_\delta)\setminus \mathcal P \\
      t_v(v) &\mbox{if } v\in \mathcal P \\
      t_w(v) &\mbox{else, i.e., } v\in W(G^*)\setminus V^0(T_\delta)
      \text{ and thus, } v \text{ is a vertex of } G_w \text{ for some } w\in P 
    \end{cases}
  \end{equation*}
\end{enumerate}
\end{definition}
We next derive some basic properties for pvr graphs that we need later in
order to show that pvr graphs can explain a given map $\delta$.

\begin{figure}[t]
  \begin{center}
    \includegraphics[width=0.5\textwidth]{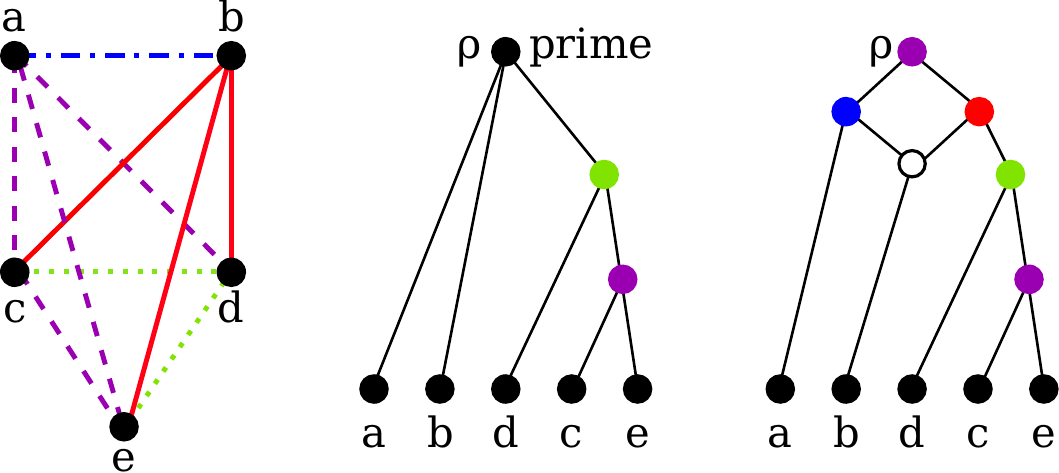}
  \end{center}
  \caption{Left: The graph-representation of the map
    $\delta\colon\Xirr\to\Upsilon$ as in Fig.\ \ref{fig:quotient}. Middle:
    The MDT $(T_\delta,t_\delta)$. Right: A pvr graph $(G^*,t^*)$ of
    $(T_\delta,t_\delta)$.  The root $\rho$ of $(T_\delta,t_\delta)$ is
    labeled ``prime'', since $L_{T_\delta}(\rho)=X$ and
    $\delta_{|X}/\Mmax(\delta_{|X}) = \delta/\Mmax(\delta)$ is prime, as
    outlined in the caption of Fig.\ \ref{fig:quotient}. The pvr graph
    $(G^*,t^*)$ is obtained from $(T_\delta,t_\delta)$ by replacing the
    root $\rho$ by the median graph $\Hext_{3}$ that explains
    $\delta_{|X}/\Mmax(\delta_{|X})$ (represented in Fig.\
    \ref{fig:quotient}, right). The labeling of the white vertex in
    $(G^*,t^*)$ can be chosen arbitrarily. }
  \label{fig:MDT}
\end{figure}

\begin{lemma}\label{lem:some-properties1}
  Let $\delta\colon \Xirr\to\Upsilon$. Then, the pvr graph $(G^*,t^*)$ of
  the MDT $(T_\delta,t_\delta)$ constructed according to Def.\
  \ref{def:pvr} is well-defined and unique up to the choice of the median
  graphs $(G_v,t_v)$ in Def.\ \ref{def:pvr}\eqref{step:Gv}.  Furthermore,
  we have $V(T_\delta)\subseteq V(G^*)$.
\end{lemma}
\begin{proof}
  Let $(T_\delta,t_\delta)$ be the MDT of $\delta$ and let $L=X$ denote the
  leaf set of $ T_\delta$.  Let $\mathcal{P}$ be the set of all prime
  vertices in $T_\delta$.  We show first that $(G^*,t^*)$ is
  well-defined. By construction, $L\subseteq V(G^*)$. Moreover, 
  if $v\in V^0(T_\delta)$ is a non-prime vertex,
   then it still exists in
  $G^*$ and $v\in W(G^*)$. Thus,
  we can put $t^*(v)=t_\delta(v)$. This part is clearly well-defined.

  Now let $v\in V(T_\delta)$ be a prime vertex.  By definition,
  $\delta_v\coloneqq \delta_{|L(v)}/\Mmax(\delta_{|L(v)})$ consists of
  trivial modules only.  Since $\Mmax(\delta_{|L(v)})$ is a subset of
  strong modules of $\delta$ and $\delta$ is a symmetric map, any two
  distinct modules $M, M'\in \Mmax(\delta_{|L(v)})$ must be
  disjoint. Hence, Lemma \ref{lem:arcs-modules} implies that there is a
  unique label $i\in \Upsilon$ such that $\delta(x,y) = \delta(y,x) = i$
  for all $x \in M$ and $y \in M'$. By definition, we thus have
  $\delta_v(M,M') = \delta_v(M',M) = i$ for some unique label
  $i\in \Upsilon$.  Consequently, $\delta_v$ is a symmetric map and
  well-defined.  By Theorem \ref{thm:median-explains} there is a labeled
  median graph $(G_v,t_v)$ that explains $\delta_v$ and such that
  $v\in V_{\med_{G_v},L(G_v)}$.  Hence, $(G_v,t_v)$ is well defined.  Note
  that each $M\in \Mmax(\delta_{|L(v)})$ corresponds to some module $L(u)$,
  $u\in \child_{T_\delta}(v)$ and that $G_v$ has leaf set
  $L(G_v) = \child_{T_\delta}(v)$ where each child
  $u\in \child_{T_\delta}(v)$ is uniquely identified with the module
  $L(u)$.  Since $v$ is prime, the edges between the children of $v$ and
  vertex $v$ are removed and we add the median graph $G_v$ with leaf-set
  $L(G_v) = \child_{T_\delta}(v)$ by identifying its root with $v$ and
  every $u\in L(G_v)$ with the unique child $u\in \child_{T_\delta}(v)$ in
  $T'$.  As this step is uniquely determined (up to the choice of $G_v$)
  and applied precisely once to prime vertices $v$, we can conclude that
  $G^*$ is well-defined.

  These arguments in particular imply $V(T_\delta)\subseteq V(G^*)$.

  We continue to show that the labeling $t^*$ is well-defined.  As argued
  above, $t^*(v)=t_\delta(v)$ is well-defined for non-prime vertices $v$ of
  $T_\delta$.  Moreover, since $t_v$ is a map from $V_{\med_{G_v},L(G_v)}$
  to $\Upsilon$ and since $v\in V_{\med_{G_v},L(G_v)}$, the assignment
  $t^*(v)=t_v(v)$ is well defined for all $v\in \mathcal P$.  Note, none of
  the leaves $u\in L(G_v)$ are contained in $V_{\med_{G_v},L(G_v)}$ and
  thus, do not obtain a label $t_v(u)$. By construction, $u\in L(G_v)$
  implies that $u\in V(T_\delta)$. Hence, $u\in L(G_v)$ obtains either the
  unique label $t^*(u)=t_\delta(u)$ if $u\notin \mathcal{P}$ or
  $t^*(u)=t_u(u)$ if $u \in \mathcal{P}$.  In summary, for all vertices
  $V(T_\delta)$ the labeling $t^*$ is well-defined.  Now let
  $v\in W(G^*)\setminus V^0(T_\delta)$. In this case, there is a vertex $w$
  distinct from $v$ such that $v\in V(G_w)$ and $G_w$ is the median graph
  chosen in Step \ref{step:Gv}.  Now, $G_w$ has root $w$.  Since
  $v\in W(G^*)\setminus V^0(T_\delta)$ we have $w\in V_{\med_{G_w},L(G_w)}$
  by construction, and therefore $v$ is labeled by $t_w(v)$.  By
  construction of $t^*$, we have $t^*(v) = t_w(v)$, which is
  well-defined. In summary, therefore, $t^*$ is well-defined.
\end{proof}
  
\begin{lemma}\label{lem:some-properties2}
  Let $\delta\colon \Xirr\to\Upsilon$ be a map and let $(G^*,t^*)$ be a pvr
  graph of the MDT $(T_\delta,t_\delta)$.  If $u,v\in V(T_\delta)$ such
  that $u\preceq_{T_\delta}v$, then $u\preceq_{G^*} v$ and the vertices on
  the (unique) shortest path $P_{T_\delta}(u,v)$ in $T_\delta$ are
  contained in the vertex set of every path $P_{G^*}(u,v)$ in $G^*$.
\end{lemma}
\begin{proof}
  If $u=v$ in $T_\delta$, then we can apply Lemma
  \ref{lem:some-properties2} to conclude that $V(T_\delta)\subseteq V(G^*)$
  and thus, $u=v$ in $G^*$.  In this case, the path $P_{T_\delta}(u,v)$ in
  $T_\delta$ consists of $u$ only and so, $P_{G^*}(u,v)$ does.  Hence, we
  assume in the following that $u,v\in V(T_\delta)$ are chosen such that
  $u\prec_{T_\delta} v$.

  Assume that $\{u,v\}\in E(T_\delta)$. If $v$ is not a prime vertex, then
  this edge $\{u,v\}$ also exists in $G^*$, by construction.  Thus
  $u\prec_{G^*} v$.  Otherwise, if $v$ is a prime vertex it is replaced by
  a median graph $G_v$ with root $v$ and $u\in L(G_v)$. That is,
  $u\prec_{G_v} v$ and, by construction, $u\prec_{G^*} v$.
  
  Assume now that $\{u,v\}\notin E(T_\delta)$ and consider the unique path
  $P_{T_{\delta}}(u,v)$. By analogous arguments, $b\prec_{G^*} a$ for every
  edge $\{a,b\}$ in the path $P_{T_{\delta}}(a,b)$ with
  $b\prec_{T_\delta} a$.  By induction on the number of edges, we thus
  conclude that $u\prec_{G^*} v$.

  It remains to show that the vertices in the unique shortest
  $P_{T_\delta}(u,v)$ in $T_\delta$ are contained in the vertex set of
  every path $P_{G^*}(u,v)$ in $G^*$.  By Lemma \ref{lem:some-properties1},
  the vertices in $P_{T_\delta}(u,v)$ are contained in $V(G^*)$.  Let
  $\{a,b\}$ be an edge in $P_{T_\delta}(u,v)$ with $b\prec_{T_\delta} a$.
  As argued above, if $a$ is not prime, then $\{a,b\}$ is an edge in $G^*$
  and, otherwise, we still have $b\prec_{G^*} a$.  Hence, if $a$ is not
  prime, then $P_{G^*}(u,v)$ must contain the edge $\{a,b\}$ (since $a$ is
  the unique last ancestor of $b$) and, if $a$ is prime, then
  $P_{G^*}(u,v)$ must contain a subpath from $a$ to $b$ in ${G^*}$ that
  starts at the root $a$ of $G_a$ and ends in the leaf $b$ of $G_a$. In
  summary, $P_{G^*}(u,v)$ either still contains the edge $\{a,b\}$ or a
  path from $a$ to $b$ in $G^*$. Hence, the vertices $a$ and $b$ are
  contained $P_{G^*}(u,v)$. Since the choice of the edge in
  $P_{T_\delta}(u,v)$ was arbitrary, all vertices of $P_{T_\delta}(u,v)$
  are contained in the vertex set of $P_{G^*}(u,v)$. Since these arguments
  apply to all paths $P_{G^*}(u,v)$, the statement follows.
\end{proof}

\begin{lemma}\label{lem:some-properties3}
  Let $\delta\colon \Xirr\to\Upsilon$ be a map and let $(G^*,t^*)$ be a pvr
  graph of the MDT $(T_\delta,t_\delta)$. Let $x,y\in X$ be distinct and
  denote by $v_x$ and $v_y$ the two children of
  $v\coloneqq\lca_{T_\delta}(x,y)$ with $x\preceq_{T_\delta}v_x$ and
  $y\preceq_{T_\delta} v_y$, respectively. Then, the following two
  statement are true:
  \begin{itemize}
  \item[(i)] $v_x,v_y\in  P_{G^*}(x,y)$.
  \item[(ii)] $\med_{G^*}(\rho,x,y)= \med_{G^*}(v,v_x,v_y)$. Moreover,
    $\med_{G^*}(v,v_x,v_y) = \med_{G_v}(v,v_x,v_y)$ in case $v$ is prime
    and $G_v$ is chosen in Step \eqref{step:Gv} in Def.\ \ref{def:pvr}.
  \end{itemize}
\end{lemma}
\begin{proof}
  (i) If $v =\lca_{T_\delta}(x,y)$ is not a prime vertex, then every path
  between $v$ and $x$ in $G^*$ contains $v_x$ while every path between $v$
  and $y$ in $G^*$ contains $v_y$.  Moreover, every path between $x$ and
  $y$ in $G^*$ must contain vertex $v$. Since every path between $v$ and
  $x$ in $G^*$ as well as between $v$ and $y$ in $G^*$ is a subpath of some
  path between $x$ and $y$, the vertices $v_x$ and $v_y$ are contained in
  $P_{G^*}(x,y)$.  If $v$ is a prime vertex, then it is replaced by a
  median graph $G_v$ and, by construction $v_x$ and $v_y$ are leaves in
  $G_v$. Hence, by construction, $v_x$ and $v_y$ are
  incomparable. Moreover, by Lemma~\ref{lem:some-properties2} we have
  $x\preceq_{G^*}v_x$ and $y\preceq_{G^*}v_y$.  Note, by construction,
  there is no vertex that is incomparable to $v_x$ (resp., $v_y$) that is
  also an ancestor of $x$ (resp., $y$).  Consequently, $v_x$ and $v_y$ are
  contained in $P_{G^*}(x,y)$.

  (ii) Application of Lemma~\ref{lem:some-properties2} implies that
  $v,v_x\in P_{T_\delta}(\rho,x)\subseteq P_{G*}(\rho,x)$ and
  $v,v_y\in P_{T_\delta}(\rho,y)\subseteq P_{G^*}(\rho,y)$. Moreover, (i)
  implies $v_x,v_y\in P_{T_\delta}(x,y)\subseteq P_{G^*}(x,y)$.  Hence,
  $I_{G^*}(v,v_x)\subseteq I_{G^*}(\rho,x)$,
  $I_{G^*}(v,v_y)\subseteq I_{G^*}(\rho,y)$ and
  $I_{G^*}(v_x,v_y)\subseteq I_{G^*}(x,y)$.  Therefore,
  $I_{G^*}(v,v_x)\cap I_{G^*}(v,v_y)\cap I_{G^*}(v_x,v_y)\subseteq
  I_{G^*}(\rho,x)\cap I_{G^*}(\rho,y)\cap I_{G^*}(x,y)$. Since $G^*$ is a
  median graph, both intersections have precisely one element namely
  $\med_{G^*}(v,v_x,v_y)$ and $\med_{G^*}(\rho,x,y)$, which therefore are
  identical.  By construction, furthermore, we also have
  $\med_{G^*}(v,v_x,v_y) = \med_{G_v}(v,v_x,v_y)$ in case $v$ is prime.
\end{proof}

\begin{theorem}
  Let $\delta$ be a symmetric map with MDT $(T_\delta,t_\delta)$ and let
  $(G^*,t^*)$ be a pvr graph for $(T_\delta,t_\delta)$.  Then, $(G^*,t^*)$
  is a median graph that explains $\delta$.
\label{thm:pvr}
\end{theorem}
\begin{proof}
  Let $\mathcal{P}$ be the set of all prime vertices in $T_\delta$.  We
  show first that $G^*$ is a median graph.  To this end, consider the graph
  $G_v-L(G_v)$ obtained from the labeled median graph $(G_v,t_v)$
  constructed in Step \eqref{step:Gv} in Def.\ \ref{def:pvr}. By Lemma
  \ref{lem:leaf-append}, $G_v-L(G_v)$ remains a median graph.  We denote
  with $H$ the graph that is the disjoint union of the graphs $G_v-L(G_v)$
  (for all prime vertices $v$) and the connected components of the forest
  $T'$ obtained in Step \eqref{step:T'} in Def.\ \ref{def:pvr}.  Note, all
  connected components of $H$ are median graphs and, in particular, $H$ is
  a spanning subgraph of $G^*$.  It is an easy task to verify that every
  edge $e\in E(G^*)\setminus E(H)$ connects precisely two connected
  components of $H$. Moreover, all edges $e\in E(G^*)\setminus E(H)$ that
  are added to obtain $G^*$ are only edges that are already contained in
  $T_\delta$. In other words, stepwise addition of edges
  $e\in E(G^*)\setminus E(H)$ cannot create new cycles in $G^*$. Hence,
  whenever we have added a proper subset $F\subset E(G^*)\setminus E(H)$ of
  edges to $H$ and take a further edge $f\in E(G^*)\setminus (E(H)\cup F)$
  then it must connect again two connected components of the graph
  $(V(G^*), E(H)\cup F)$. By induction and Lemma \ref{lem:leaf-append}, the
  connected components that are joined by a new edge $f$ are again median
  graphs.  As a consequence, $G^*$ is a median graph.

  Let $\rho$ be the root of $G^*$.  It remains to show that $(G^*,t^*)$
  explains $\delta$. Recall that we have $L(G^*)=X$ by construction. In
  order to verify that $\delta(x,y) = t^*(\med_{G^*}(\rho,x,y))$ for all
  distinct $x,y\in X$ we consider the following cases: Either
  $t_\delta(\lca_{T_\delta}(x,y)) = i$ for some $i\in \Upsilon$ or
  $t_\delta(\lca_{T_\delta}(x,y)) = \text{prime}$.  Set
  $v\coloneqq \lca_{T_\delta}(x,y)$ and denote by $v_x$ and $v_y$ the two
  children of $v$ with $x\preceq_{T_\delta} v_x$ and
  $y\preceq_{T_\delta} v_y$.  Note, $v_x$ and $v_y$ are incomparable in
  ${T_\delta}$ since $v = \lca_{T_\delta}(x,y)$.  By Lemma
  \ref{lem:some-properties2}, the vertices $v_x$ and $v_y$ still exist in
  $G^*$ and are, by construction still incomparable in $G^*$.

  Suppose first $t_\delta(v) = i$ for some $i\in \Upsilon$ in which case
  $\delta(x,y)=i$. In this case, $v$ is not a prime vertex and, by
  construction, its two children in $G^*$ are still $v_x$ and $v_y$. By
  Lemma \ref{lem:some-properties2}, $v$ is contained in the shortest paths
  $P_{G^*}(\rho,x)$ as well as $P_{G^*}(\rho,y)$. Moreover, by Lemma
  \ref{lem:some-properties3}(i), $P_{G^*}(x,y)$ contains $v_x$ and $v_y$.
  Since $v_x$ and $v_y$ have only vertex $v$ as common adjacent vertex and
  since any path connecting $v_x$ and $v_y$ contains $v$, we can conclude
  that $P_{G^*}(x,y)$ must contain $v$. Since $G^*$ is a median graph this
  implies that $\med_{G^*}(\rho,x,y) = v$.  Since $v$ is not a prime
  vertex, we have, by construction,
  $t^*(\med_{G^*}(\rho,x,y))= t_{\delta}(v) = \delta(x,y)$.

  Now assume that $v$ is prime.  In this case, $v$ has been replaced by the
  median graph $G_v$ according to Def.\ \ref{def:pvr}\eqref{step:Gv}.  In
  particular, $(G_v,t_v)$ explains
  $\delta_v\coloneqq\delta_{|L(v)}/\Mmax(\delta_{|(L(v)})$. Note, for each
  child $u\in \child_{T_\delta}(v)$, the set $L(u)$ is a module in
  $\Mmax(\delta_{|L(v)})$.  Since $v$ is prime, $\delta_v$ consists of
  trivial modules only, that is, every $L(u)\in \Mmax(\delta_{|(L(v)})$
  forms a trivial module $M=\{L(u)\}$ in $\M(\delta_v)$. This and
  definition of $\delta_v$ implies that $\delta_v$ maps pairs $(M,M')$ of
  distinct trivial modules to some label in $\Upsilon$.  By definition of
  $L(G_v)$, the leaves $v_x$ and $v_y$ are contained in $L(G_v)$ and
  represent the trivial modules $\{L(v_x)\}$ and $\{L(v_y)\}$ of
  $\delta_v$. Since $L(v_x),L(v_y) \in \Mmax(\delta_{|(L(v)})$, they are
  disjoint and hence, Lemma \ref{lem:arcs-modules} implies that there is a
  unique label $i\in \Upsilon$ such that such that $\delta(a,b)=i$ for all
  $a\in L(v_x)$ and all $b\in L(v_y)$.  Therefore we have, by definition,
  $\delta_v(\{L(v_x)\},\{L(v_y\})) = i = \delta(x,y)$.  Since $v_x$ and
  $v_y$ are contained in $L(G_v)$ and represent the trivial modules
  $\{L(v_x)\}$ and $\{L(v_y)\}$ of $\delta_v$, and since $(G_v,t_v)$
  explains $\delta_v$, we have
  $t_v(\med_{G_v}(\rho, v_x,v_y)) = \delta(x,y)$.  By Lemma
  \ref{lem:some-properties3}(ii),
  $\med_{G^*}(\rho,x,y) = \med_{G_v}(v,v_x,v_y) =\med_{G^*}(v,v_x,v_y)$.
  This and Def.\ \ref{def:pvr}\eqref{step:color} implies
  $t^*(\med_{G^*}(\rho,x,y))= t_v(\med_{G_v}(\rho, v_x,v_y)) =
  \delta(x,y)$.
\end{proof}

As a direct consequence of Thm.~\ref{thm:pvr} we obtain a practical
algorithm to compute a pvr graph for a given map $\delta$ that is linear in
the size of the input.

\begin{algorithm}[t]
  \caption{Construction of a labeled median graph that explains a given
    symmetric map.}
  \label{alg:pvr}
\begin{algorithmic}[1]
\Require symmetric map $\delta\colon    \Xirr\to\Upsilon$
\Ensure pvr graph $(G^*,t^*)$ that explains $\delta$
\State Compute MDT $(T_\delta, t_\delta)$
\State $(G^*,t^*)\gets (T_\delta, t_\delta)$
\State $\mathcal{P}\gets$ set of prime vertices in $(T_\delta, t_\delta)$
\For{all $v\in \mathcal{P}$}
   \State $n_v\gets|\child_{T_\delta}(v)|$
   \State $(G^*,t^*) \gets$ graph obtained from $(G^*,t^*)$ by
   replacing  $v$ by $(G_v,t_v)$ where $G_v \simeq \Hext_{n_v}$
   (cf.\ Def.\ \ref{def:pvr})
\EndFor
\State \Return  $(G^*,t^*)$
\end{algorithmic}
\end{algorithm}

\begin{theorem}
  Algorithm \ref{alg:pvr} correctly computes a labeled median graph $(G,t)$
  that explains a given symmetric map $\delta\colon \Xirr\to\Upsilon$ in
  $O(m)$ time, where the input size is $m=|X^2|$.
\end{theorem}
\begin{proof}
  By Theorem \ref{thm:pvr}, every pvr graph as constructed in Def.\
  \ref{def:pvr} explains $\delta$. Hence, it remains to show that replacing
  each prime vertex $v$ by $(G_v,t_v) = (\Hext_{n},t_v)$ with
  $n_v=|\child_{T_\delta}(v)|$ yields a valid pvr graph. Observe first that
  every prime vertex $v$ must have at least three children since otherwise,
  $\delta_{|L(v)}/\Mmax(\delta_{|L(v)})$ is complete.  Hence, $\Hext_{n_v}$
  is well-defined since $n_v\geq 3$ for all $v\in \mathcal{P}$.  This and
  Prop.\ \ref{prop:ext-hg-median} implies that $\Hext_{n_v}$ is a
  well-defined median graph for all $v\in \mathcal{P}$. As outlined in the
  proof of Thm.\ \ref{thm:median-explains}, there is a labeling $t_v$ such
  that $(\Hext_{n_v}, t_v)$ explains any map
  $\delta_{|L(v)}/\Mmax(\delta_{|L(v)})$.  Hence, Algorithm \ref{alg:pvr}
  is correct.

  To compute the running time, we first note that by \cite[Thm.\ 7]{HSW:16}
  the MDT $(T_\delta,t_\delta)$ can be computed in $O(|X|^2)$ time.
  Computing the initial graph $(G^*,t^*)$ and the set of prime vertices
  $\mathcal{P}$ can be done in $O(|X|)$ because $T_{\delta}$ is a tree and
  thus has $O(|X|)$ edges and vertices. Set $V\coloneqq V(T_\delta)$.  By
  construction, $\Hext_{n_v}$ has $O(n_v^2)$ edges and vertices.  Each
  prime vertex $v\in \mathcal{P}$ therefore can be replaced by $(G_v,t_v)$
  with $G_v \simeq \Hext_{n_v}$ in $O(n_v^2)$ time. Repeating the latter
  for all prime vertices works thus in
  $O(\sum_{v \in \mathcal{P}} n_v^2) \subseteq O(\sum_{v \in V}n_v^2)
  \subseteq O(|V|\sum_{v \in V}n_v) \subseteq O(|V|\sum_{v \in V}\deg(v))
  \subseteq O(|V||E(T_\delta)|) = O(|V|^2) = O(|X|^2)$ time. The total
  effort is therefore in $O(|X|^2)$.
\end{proof}

\section{Future Directions}

In summary, we have shown that labeled extended hypercubes, extended
half-grids, and pvr graphs can be used to explain a given symmetric map
$\delta$.  Clearly, pvr graphs with leaf set $X$ have usually fewer
vertices than extended half-grids $\Hext_{|X|}$, which in turn are much
smaller than extended hypercubes $\Qext_{|X|}$. This begs the question
under which conditions it is possible to further simplify a rooted median
graph $(G,t)$ that explains a map $\delta$. More precisely, a given labeled
rooted median graph $(G,t)$ with leaf set $X$ is \emph{least-resolved}
w.r.t.\ a given map $\delta\colon \Xirr\to\Upsilon$ if it explains $\delta$
and there is no labeled rooted median graph $(G^*,t^*)$ explaining $\delta$
that can be obtained from $G$ by edge contraction (and the removal of any
multi-edge that may result in the process). Furthermore, let us say that
$(G,t)$ is \emph{minimally-resolved} w.r.t.\ $\delta$, if it explains
$\delta$ and has among all labeled rooted median graphs with leaf set $X$
that explain $\delta$ the fewest number of vertices.

\begin{figure}[t]
  \begin{center}
    \includegraphics[width=0.7\textwidth]{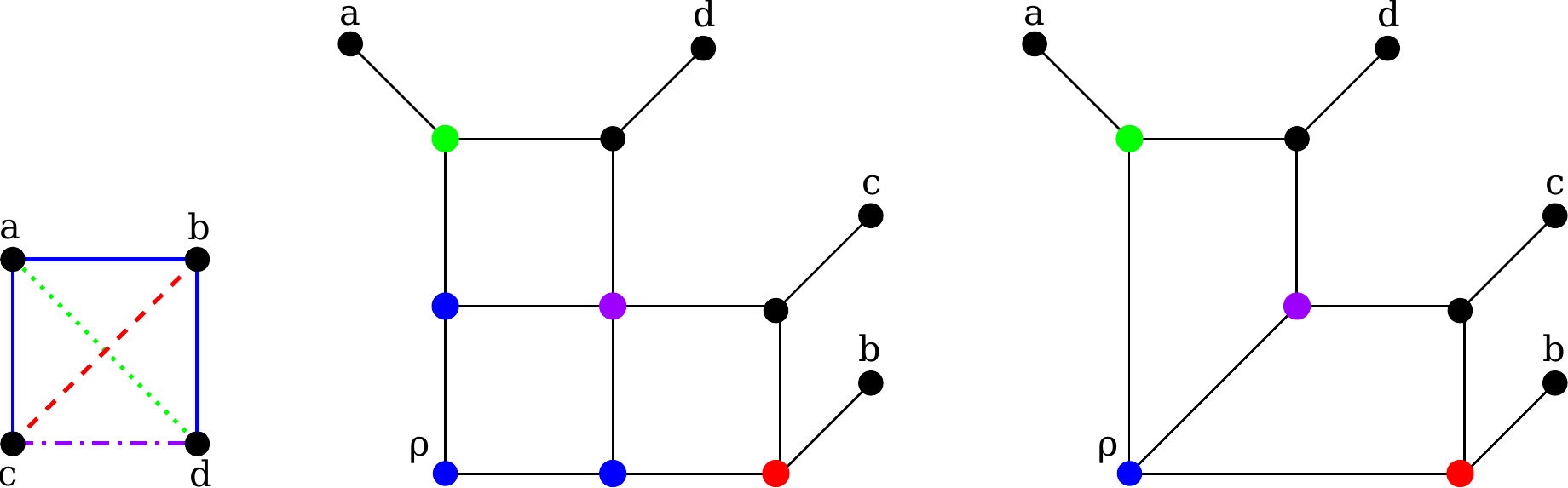}
  \end{center}
  \caption{The map $\delta$ as in Fig.\ \ref{fig:exmpl}(left) explained by
    a labeled halfgrid $(\Hext_{4},t)$ that explains $\delta$ (middle).
    After contraction of the two edges incident to the root and removing
    one of the resulting ``multi''-edges one obtains the graph $(G,t')$ on
    the right. It is easy to verify that $G$ is a grid-graph and, since all
    inner faces are squares, $G$ is a median graph (c.f.\ Thm.\
    \ref{thm:grid-median}). In particular, $(G,t')$ is a least-resolved
    median graph that explains $\delta$.}
  \label{fig:contract}
\end{figure}

\begin{figure}[t]
  \begin{center}
    \includegraphics[width=0.5\textwidth]{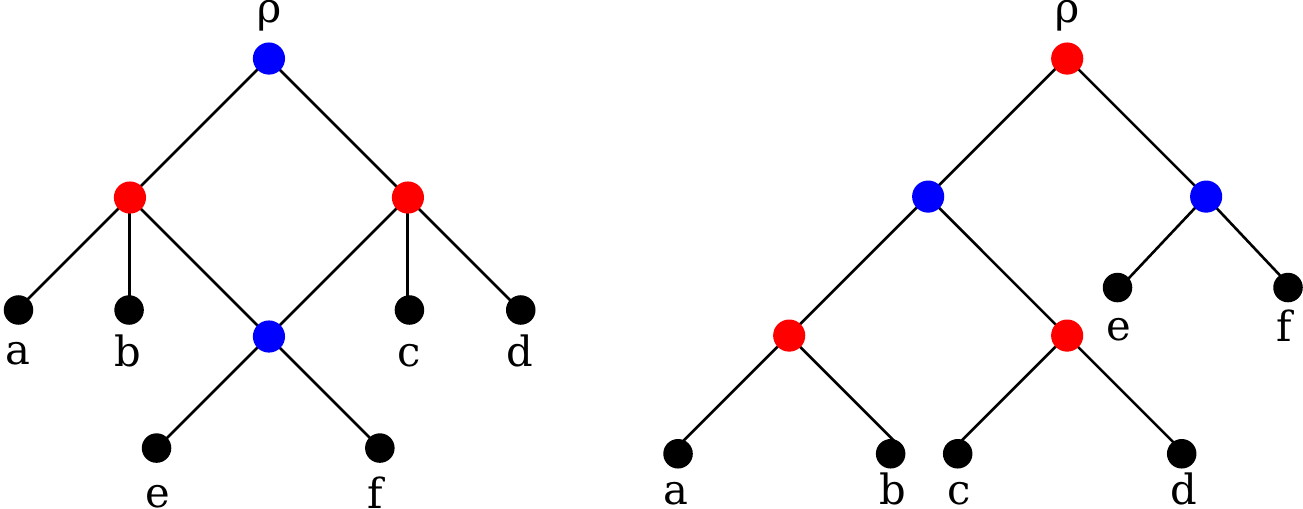}
  \end{center}
  \caption{Both graphs explain the same symmetric map $\delta$.  However,
    the rooted median graph $(G,t)$ (left) has only 10 vertices and 10
    edges while the labeled MDT $(T_\delta,t_\delta)$ has 11 vertices and
    10 edges. Thus, although $(T_\delta,t_\delta)$ is least-resolved
    w.r.t.\ $\delta$ it is not minimally-resolved w.r.t.\ $\delta$.}
  \label{fig:min}
\end{figure}

The graph in Fig.~\ref{fig:contract}(right), for example, is obtained from
the halfgrid in Fig.~\ref{fig:contract}(middle) by contraction of two edges
and removal of one of the resulting multi-edges connecting the root $\rho$
and the purple-colored vertex. The resulting graph is still a median graph
that explains $\delta$.  Hence, the halfgrid is not least-resolved w.r.t.\
the given map $\delta$. The tree $(T_{\delta},t_{\delta})$ in Fig.\
\ref{fig:min} is least-resolved w.r.t.\ $\delta$. It is not minimally
resolved, however, since the median graph $(G,t)$ in Fig.~\ref{fig:min}
explains the same map $\delta$ with fewer vertices.  These example suggest
to characterize least resolved and minimally resolved rooted median graphs
that explain a given map $\delta$. The example in Fig.~\ref{fig:min} begs
the question which symbolic ultrametrics can be explained by median graph
with fewer vertices than the MDT or cotree?

The definition of minimal resolution suggested above above uses $|V(G)|$ to
measure the size of $(G,t)$. It may no be the most natural choice,
however. As an alternative, one might want to consider median graphs with a
minimal number of edges. Are the discriminating cotrees minimal
explanations for symbolic ultrametrics if $|E(G)|$ is used to quantify
size?

A related topic for future work is the use of unrooted instead of rooted
graphs as models of evolutionary relationships. Median networks again
appear as a natural generalization of trees. In fact they are used to
describe population-level variations e.g.\ in \cite{Bandelt:00}. It seems,
however, that in the unrooted setting ternary rather than binary relations
become the natural mathematical objects to encode events and
properties. This line of reasoning has led to investigations into symbolic
ternary metrics and the characterization of a generalization of symbolic
ultrametrics \cite{Gruenewald:18,Huber:19}. In \cite{Huber:18} such
relations are considered in the context of orthology in the setting of
level-1 networks as a generalization of trees.

\paragraph{Acknowledgements.}
Stimulating discussions with Carsten R.\ Seemann are gratefully
  acknowledged. This work was funded in part by the \emph{Deutsche
    Forschungsgemeinschaft}.

\bibliographystyle{plain}
\bibliography{biblio}
\end{document}